\documentclass[a4paper,pdftex,reqno,10pt]{amsart}%

\usepackage{amssymb,amsmath}
\usepackage{mathptmx}       
\usepackage{helvet}         
\usepackage{courier}        
\usepackage{type1cm}        
\usepackage{url}
\usepackage{graphicx}        
\usepackage{multicol}        
\usepackage[bottom]{footmisc}

\usepackage{bm} 
\usepackage{enumerate} 
\usepackage{ytableau} 

\newcommand{\F}{\mathbb{F}}
\newcommand{\gauss}[3]{\genfrac{[}{]}{0pt}{}{#1}{#2}_{#3}}
\newcommand{\SPG}[2]{\operatorname{PG}(#1,#2)}
\newcommand{\PG}{\operatorname{PG}}
\newcommand{\SV}[2]{\mathbb{F}_{#2}^{#1}}

\newcommand{\aspace}{\SPG{v-1}{\mathbb{F}_q}}
\newcommand{\vek}[1]{\mathbf{#1}}
\newcommand{\mat}[1]{\mathbf{#1}}
\newcommand{\hdist}{\mathrm{d}}
\newcommand{\hweight}{\mathrm{w}}
\newcommand{\sdist}{\mathrm{d}_{\mathrm{S}}}
\newcommand{\rdist}{\mathrm{d}_{\mathrm{R}}}
\newcommand{\smax}{\mathrm{A}}
\newcommand{\imat}{\mathbf{I}}
\newcommand{\rk}{\operatorname{rk}}
\newcommand{\deficiency}{\sigma}
\newcommand{\points}{\mathcal{P}}

\newcommand{\card}[1]{\##1}
\newcommand{\frobenius}{\mathrm{F}}
\newtheorem{Construction}{Construction}

\newtheorem{theorem}{Theorem}
\newtheorem{lemma}{Lemma}
\newtheorem{corollary}{Corollary}
\newtheorem{definition}{Definition}
\newtheorem{example}{Example}

\makeindex             
\begin{document}

\title{Partial spreads and vector space partitions}

\author{Thomas Honold, Michael Kiermaier, and Sascha Kurz}
\address{Thomas Honold, Zhejiang University, 310027 Hangzhou, China}
\email{honold@zju.edu.cn}
\address{Michael Kiermaier,University of Bayreuth, 95440 Bayreuth, Germany}
\email{michael.kiermaier@uni-bayreuth.de}
\address{Sascha Kurz, University of Bayreuth, 95440 Bayreuth, Germany}
\email{sascha.kurz@uni-bayreuth.de}

\abstract{Constant-dimension codes with the maximum possible minimum
  distance have been studied under the name of partial spreads in
  Finite Geometry for several decades. Not surprisingly, for this
  subclass typically the sharpest bounds on the maximal code size are
  known. 
  The seminal works of Beutelspacher and Drake \& Freeman on partial
  spreads date back to 1975, and 1979, respectively. From then until
  recently, there was almost no progress besides some computer-based
  constructions and classifications. It turns out that vector space
  partitions provide the appropriate theoretical framework and can be
  used to improve the long-standing bounds in quite a few cases. Here,
  we provide a 
  historic account on partial spreads and an
  interpretation of the classical results from a modern perspective.  
  To this end, we introduce all required methods from the theory
  of vector space partitions and Finite Geometry in a tutorial
  style. We guide the reader to the current frontiers of research in
  that field, including a detailed description of the recent improvements.
}}
\maketitle

\section{Introduction}
Let $\mathbb{F}_q$ be the finite field with $q$ elements, where $q>1$
is a prime power.  By $\mathbb{F}_q^v$ we denote the standard vector
space of dimension $v\ge 1$ over $\mathbb{F}_q$, whose vectors are the
$v$-tuples $\vek{x}=(x_1,\dots,x_v)$ with $x_i\in\mathbb{F}_q$. The set of all
subspaces of $\mathbb{F}_q^v$, ordered by the incidence relation
$\subseteq$, is called \emph{($v-1$)-dimensional 
  projective geometry over $\mathbb{F}_q$} and denoted by
$\aspace$. It forms a finite modular geometric lattice
with meet $X\wedge Y=X\cap Y$, join $X\vee Y=X+Y$, and rank function
$X\mapsto\dim(X)$. Employing this algebraic notion of dimension instead of
the geometric one, we will use the term
\emph{$k$-subspace} to denote a $k$-dimensional vector subspace of
$\mathbb{F}_q^v$.\footnote{Using the algebraic dimension
  has certain advantages---for example, the existence criterion $v=tk$
  for spreads (cf.\ Theorem~\ref{thm_spread}) looks ugly when
  stated in terms of the geometric dimensions: $v'=t(k'-1)+1$.}
The important geometric interpretation of subspaces
will still be visible in the terms \emph{points, lines, planes,
  solids, hyperplanes} (denoting $1$-, $2$-, $3$-, $4$- and
$(v-1)$-subspaces, respectively), and in general through our
extensive use of geometric language.


In the same way as $\SV{v}{q}$,
an arbitrary $v$-dimensional vector space $V$ over $\F_q$ gives rise
to a projective geometry $\PG(V)$, and the terminology introduced
before (and thereafter) applies to this general case as well. Since a
vector space isomorphism $V\cong\F_q^v$ induces a geometric
isomorphism (``collineation'') $\PG(V)\cong\PG(\F_q^v)=\PG(v-1,\F_q)$,
we could in principle avoid the use of non-standard vector spaces, but
only at the expense of flexibility---for example, the Singer
representation of the point-hyperplane design of $\aspace$ is
best developed using the field extension $\F_{q^v}/\F_q$ as ambient vector
space $V$ (and not $\SV{v}{q}$, which would require a discussion of matrix
representations of finite fields).

The set of all $k$-subspaces of an $\F_q$-vector space $V$
will be denoted by $\gauss{V}{k}{q}$. The sets $\gauss{V}{k}{q}$ form
finite analogues of the Gra\ss mann varieties studied in Algebraic
Geometry.  In terms of $v = \dim(V)$, the cardinality of
$\gauss{V}{k}{q}$ is given by the Gaussian binomial coefficient
$$
\gauss{v}{k}{q} :=
\begin{cases}
	\frac{(q^v-1)(q^{v-1}-1)\cdots(q^{v-k+1}-1)}{(q^k-1)(q^{k-1}-1)\cdots(q-1)} & \text{if }0\leq k\leq v\text{;}\\
	0 & \text{otherwise,}
\end{cases}
$$
which are polynomials of degree $k(v-k)$ in $q$ (if they are nonzero)
and represent $q$-analogues of the ordinary binomial coefficients in
the sense that $\lim_{q\to 1}\gauss{v}{k}{q}=\binom{k}{k}$. Their most
important combinatorial properties are described in
\cite[Sect.~3.3]{andrews1998theory} and \cite[Ch.~24]{lint-wilson92}.

Making the connection with the main topic of this book, the geometry
$\aspace$ serves as input and output alphabet of the so-called
\emph{linear operator channel (LOC)}, a clever model for information
transmission in coded packet networks subject to noise
\cite{koetter-kschischang08}.\footnote{The use of distributed coding
  at the nodes of a packet-switched network, generally referred to as
  \emph{Network Coding}, is described in
  \cite{guang-zhang14,medard-sprintson12,yeung-li-cai06} and elsewhere
  in this volume.}  The relevant metrics on the LOC are given by the
\emph{subspace distance}
$ d_S(X,Y):=\dim(X+Y)-\dim(X\cap Y)=2\cdot\dim(X+Y)-\dim(X)-\dim(Y)$,
which can also be seen as the graph-theoretic distance in the Hasse
diagram of $\aspace$, and the \emph{injection distance}
$ d_I(X,Y):=\max\left\{\dim(X),\dim(Y)\right\}-\dim(X\cap Y) $.  A set
$\mathcal{C}$ of subspaces of $\F_q^v$ is called a \emph{subspace
  code} and serves as a channel code for the LOC in the same way as
classical linear codes over $\F_q$ do for the $q$-ary symmetric
channel.\footnote{except that attention is usually restricted to
  ``one-shot subspace codes'', i.e.\ subsets of the alphabet, which
  makes no sense in the classical case} The \emph{minimum (subspace)
  distance} of $\mathcal{C}$ is given by
$d = \min\{d_S(X,Y) \mid X,Y\in\mathcal{C}, X \neq Y\}$.  If all
elements of $\mathcal{C}$ have the same dimension, we call
$\mathcal{C}$ a \emph{constant-dimension code}.  For a
constant-dimension code $\mathcal{C}$ we have $d_S(X,Y) = 2d_I(X,Y)$
for all $X,Y\in\mathcal{C}$, so that we can restrict attention to the
subspace distance. Constant-dimension codes are the most suitable for
coding purposes, and the quest for good system performance leads
straight to the problem of determining the maximum possible cardinality
$\smax_q(v,d;k)$ of a constant-dimension-$k$ code in $\F_q^v$
with minimum subspace distance $d$. For
two codewords $X$ and $Y$ of dimension $k$ the inequality
$d_S(X,Y)\geq d$ is equivalent to
$\dim(X\cap Y)\le k-d/2$.\footnote{Note that the distance
  between codewords of the same dimension, and hence also the minimum
  distance of a constant-dimension code, is an even integer.} Thus, the
maximum possible minimum distance of a constant-dimension code with
codewords of dimension $k$ is $2k$. This extremal case has been
studied under the name ``partial spreads'' in Finite
Geometry for several decades. A \emph{partial $k$-spread} in $\F_q^v$
is a collection of $k$-subspaces with pairwise trivial,
i.e., zero-dimensional intersection. 
Translating this notion into Projective Geometry and identifying
thereby, as usual, subspaces of $\F_q^v$ with their sets of incident
points, we have that a partial $k$-spread in $\F_q^v$ is the same as
a set of mutually disjoint $k$-subspaces, or ($k-1$)-dimensional flats
in the geometric view, of the geometry $\aspace$.\footnote{In other
  words, partial spreads are just packings of the point set of a
  projective geometry $\aspace$ into subspaces of equal dimension.}

With the history of partial spreads in mind, it comes as no surprise
that the sharpest bounds on the maximal code sizes $\smax_q(v,d;k)$ of
constant-dimension codes are typically known for this special
subclass. The primary goal of this survey is to collect all available
information on the numbers $\smax_q(v,2k;k)$ and present this
information in an accessible way. Following standard practice in
Finite Geometry, we will refer to partial $k$-spreads of size
$\smax_q(v,2k;k)$ as \emph{maximal partial $k$-spreads}.\footnote{The
  weaker property of \emph{complete} (i.e., inclusion-maximal) partial
  spreads will not be considered.}

In the case of a perfect packing, i.e., a partition of the point set
of $\aspace$, we speak of a \emph{$k$-spread}. Partitions into
subspaces of possibly different dimensions are equivalent to vector
space partitions. A \emph{vector space partition} $\mathcal{C}$ of
$\SV{v}{q}$ is a collection of nonzero subspaces with the property
that every non-zero vector is contained in a unique member of
$\mathcal{C}$. If $\mathcal{C}$ contains $m_d$ subspaces of dimension
$d$, then $\mathcal{C}$ is said to be of type
$k^{m_k}\dotsm 1^{m_1}$. Zero frequencies $m_d=0$ are usually
suppressed.\footnote{Since $\sum_{X\in\mathcal{C}}\dim(X)=\sum_ddm_d=v$, the
type of $\mathcal{C}$ can be viewed as an ordinary integer partition of
$v$.}
So, partial $k$-spreads are just the special case of vector space
partitions, in which all members have dimension either $k$ or
$1$. For $k\geq 2$ (the case $k=1$ is trivial)
the members of dimension $1$ correspond to points not covered by a
$k$-subspace of the partial spread and are called \emph{holes} in this
context.

Although vector space partitions can be seen as a mixed-dimension
analogue of partial spreads, they are not usable as subspace codes of
their own.\footnote{The subspace distance $d_S(X,Y)$ depends not only
  on $\dim(X\cap Y)$ but also on $\dim(X)$ and $\dim(Y)$, which are
  not constant in this case.}  However, it turns out that they provide
an appropriate framework to study bounds on the sizes of partial
spreads.

There is a vast amount of related work that we will not cover in this
survey: Partial spreads have also been studied for combinatorial
designs and in polar spaces; for the latter see,
e.g.,~\cite{beule-klein-metsch11,eisfeldt}. In the special case $v=2k$
spreads can be used to define translation planes and provide a rich
source for constructing non-desarguesian projective planes
\cite{hughes-piper73,johnson-jha-biliotti07,lavrauw-polverino11}.
Also motivated by this geometric point-of-view, partial $k$-spreads in
$\F_q^{2k}$ of size close to the maximum size (given by
Theorem~\ref{thm_spread}) have been studied extensively. Most of this
research has focused on partial spread replacements and complete
partial spreads, while we consider only partial spreads of maximum
cardinality and hence do not touch the case $v=2k$ (except for
Theorem~\ref{thm_spread}). The classification of all (maximal) partial
spreads up to isomorphism, see e.g.~\cite{MR2475427}, is also not
treated here. Further, there is a steady stream of literature that
characterizes the existing types of vector space partitions in
$\mathbb{F}_2^v$ for small dimensions $v$. Here, we touch only briefly
on some results that are independent of the ambient space dimension
$v$ and refer to \cite{heden2012survey} otherwise.

The remaining part of this chapter is structured as follows. In
Section~\ref{sec_classical} we review some, mostly classical, bounds
and constructions for partial spreads. After introducing the concept
of $q^r$-divisible sets and codes in
Section~\ref{section_q_r_divisible}, we are able to obtain improved
upper bounds for partial spreads in Theorems~\ref{main_theorem_1}
and~\ref{main_theorem_2}. Constructions for $q^r$-divisible sets are
presented in Section~\ref{sec_constructions}, some non-existence
results for $q^r$-divisible sets are presented in
Section~\ref{sec_non_existence}, and we close this survey with a collection of
open research problems in Section~\ref{sec_open_problems}.

\section{Bounds and constructions for partial spreads}
\label{sec_classical}

Counting points in $\F_q^v$ and $\F_q^k$ gives the obvious upper
bound
$\smax_q(v,2k;k)\le
\gauss{v}{1}{q}/\gauss{k}{1}{q}=\left(q^v-1\right)/\left(q^k-1\right)$
for the size of a partial $k$-spread in $\F_q^v$. Equality corresponds
to the case of spreads, for which a handy existence criterion is known
from the work of Segre in 1964.\footnote{Segre in turn built to some
  extent on work of Andr\'e, who had earlier 
  considered the special case $v=2k$ in his seminal paper on
  translation planes~\cite{andre1954nicht}.}

\begin{theorem}[{\cite[\S VI]{segre1964teoria}, 
  \cite[p.~29]{dembowski2012finite}}]
  \label{thm_spread}
  $\mathbb{F}_q^v$ contains a $k$-spread if and only if $k$ is a
  divisor of $v$.
\end{theorem}

Since $\frac{q^v-1}{q^k-1}$ is an integer if and only if $k$ divides
$v$ (an elementary number theory exercise), only the constructive part
needs to be shown.
To this end we write $v=kt$ for a suitable integer $t$, take the
ambient space $V$ as the restriction (``field reduction'') of
$(\F_{q^k})^t$ to $\F_q$, which clearly has dimension $v$, and
define the $k$-spread $\mathcal{S}$ in $V/\F_q$ as the set of $1$-subspaces of 
$V/\F_{q^k}$. That $\mathcal{S}$ is indeed a $k$-spread,
is easily verified: Each member of $\mathcal{S}$ has dimension $k$
over $\F_q$; the members form a vector space partition of $V$ (this
property does not depend on the particular field of scalars); and the
size $\gauss{t}{1}{q^k}=\frac{q^v-1}{q^k-1}$ of $\mathcal{S}$ is as
required.\footnote{Alternatively, the member of $\mathcal{S}$
  containing a nonzero vector $\vek{x}$ is the $k$-subspace
  $\F_{q^k}\vek{x}$ of $V/\F_q$.}

\begin{example}
  \label{ex:q=2,v=4,k=2}
We consider the parameters $q=3$, $v=4$, and $k=2$. Using canonical
representatives in $\F_9\simeq \F_3[x]/(x^2+1)$, the
$\gauss{2}{1}{9}=10$ points in $\F_9^2$ are generated by
$$
  \begin{pmatrix}0\\1\end{pmatrix},
  \begin{pmatrix}1\\0\end{pmatrix},
  \begin{pmatrix}1\\1\end{pmatrix},
  \begin{pmatrix}1\\2\end{pmatrix},
  \begin{pmatrix}1\\x\end{pmatrix},
  \begin{pmatrix}1\\x+1\end{pmatrix},
  \begin{pmatrix}1\\x+2\end{pmatrix},
  \begin{pmatrix}1\\2x\end{pmatrix},
  \begin{pmatrix}1\\2x+1\end{pmatrix},
  \begin{pmatrix}1\\2x+2\end{pmatrix}.
$$
The particular 
point $P=\F_9\left(\begin{smallmatrix}1\\x+1\end{smallmatrix}\right)
=\left\{
  \left(\begin{smallmatrix}0\\0\end{smallmatrix}\right),
  \left(\begin{smallmatrix}1\\x+1\end{smallmatrix}\right),
  \left(\begin{smallmatrix}2\\2x+2\end{smallmatrix}\right),
  \left(\begin{smallmatrix}x\\x+2\end{smallmatrix}\right),
  \left(\begin{smallmatrix}x+1\\2x\end{smallmatrix}\right),
  \left(\begin{smallmatrix}x+2\\1\end{smallmatrix}\right),\right.$ \linebreak[4] 
\noindent  
  $\left.
  \left(\begin{smallmatrix}2x\\2x+1\end{smallmatrix}\right),
  \left(\begin{smallmatrix}2x+1\\x\end{smallmatrix}\right),
  \left(\begin{smallmatrix}2x+2\\2\end{smallmatrix}\right)
\right\}$
on the projective line $\SPG{1}{\F_9}$ defines a $2$-subspace of
$\F_9^2/\F_3\cong\F_3^4$, whose $8$ associated points are
$\F_3\vek{x}$, $\vek{x}\in P$,
$\vek{x}\neq\left(\begin{smallmatrix}0\\0\end{smallmatrix}\right)$;
and similarly for the other points of $\SPG{1}{\F_9}$. These ten
$2$-subspaces form the $2$-spread $\mathcal{S}$.

Using any $\F_3$-isomorphism $\F_9^2/\F_3\cong\F_3^4$, we can
translate $\mathcal{S}$ into a $2$-spread  $\mathcal{S}'$ of the
standard vector space
$\F_3^4$. Taking, for example, coordinates with respect to the basis
$(1,x)$ of $\F_9/\F_3$ and extending to $\F_9^2$ in the obvious
way translates $P$ into the $2$-subspace of $\F_3^4$ with vectors
$$
\begin{pmatrix}0\\0\\0\\0\end{pmatrix},
\begin{pmatrix}1\\0\\1\\1\end{pmatrix},
\begin{pmatrix}2\\0\\2\\2\end{pmatrix},
\begin{pmatrix}0\\1\\2\\1\end{pmatrix},
\begin{pmatrix}1\\1\\0\\2\end{pmatrix},
\begin{pmatrix}2\\1\\1\\0\end{pmatrix},
\begin{pmatrix}0\\2\\1\\2\end{pmatrix},
\begin{pmatrix}1\\2\\0\\1\end{pmatrix},
\begin{pmatrix}2\\2\\2\\0\end{pmatrix}.
$$  
The other members of $\mathcal{S}'$ are obtained in the same way.
\end{example}

From now on
we assume that $k$ does not divide $v$ and write $v=tk+r$
with $1\leq r\leq k-1$. Since the cases $t\in\{0,1\}$ are trivial
($\smax_q(r,2k;k)=0$ and $\smax_q(k+r,2k;k)=1$), we also assume
$t\geq 2$. The stated upper bound then takes the form
\begin{equation}
  \label{eq:1st_ub}
  \smax_q(v,2k;k)\leq\left\lfloor\frac{q^v-1}{q^k-1}\right\rfloor
  =\frac{q^{tk+r}-q^r}{q^k-1}+\left\lfloor\frac{q^r-1}{q^k-1}\right\rfloor
  =\sum_{s=0}^{t-1}q^{sk+r}=q^r\gauss{t}{1}{q^k}.
\end{equation}
We also see from this computation that the number of holes of a
partial $k$-spread is at least
$\frac{q^v-1}{q-1}\bmod\frac{q^k-1}{q-1}
=\frac{q^r-1}{q-1}$.
However, as we will see later, this bound can be
improved further.

In accordance with \eqref{eq:1st_ub} we make the following definition
(similar to that in \cite{beutelspacher1975partial}): The number
$\deficiency$ defined by
$\smax_q(v;2k;k)=\sum_{s=0}^{t-1}q^{sk+r}-\deficiency$ is called the
\emph{deficiency} of the maximal partial $k$-spreads in
$\F_q^v$.\footnote{This makes sense also for $r=0$:
  Spreads are assigned deficiency $\deficiency=0$.} From
\eqref{eq:1st_ub} we have $\deficiency\geq 0$. In terms of the
deficiency, the minimum possible number of holes is
$\deficiency\cdot\frac{q^k-1}{q-1}+\frac{q^r-1}{q-1}$.

Our next goal is to derive a good lower bound for $\smax_q(v,2k;k)$
(equivalently, a lower bound for the corresponding deficiency) by
constructing a large partial $k$-spread in $\F_q^v$. For this we will
employ a special case of the \emph{echelon-Ferrers
  construction}
for general subspace codes~\cite{etzion2009error}, which
involves only standard maximum rank distance codes of full row rank.

To this end, recall that every $k$-subspace $X$ of $\F_q^v$ is the row
space of a unique ``generating'' matrix $\mat{A}\in \F_q^{k\times v}$ in reduced
row-echelon form, which can be obtained by applying the Gaussian
elimination algorithm to an arbitrary generating matrix of $X$. This
matrix $\mat{A}$ is called \emph{canonical matrix} of $X$, and is uniquely
specified by its $k$ pivot columns $1\leq j_1<j_2<\dots<j_k\leq v$
(forming a $k\times k$ identity submatrix of $\mat{A}$) and the
complementary submatrix $\mat{B}\in\F_q^{k\times(v-k)}$, which 
has zero entries in positions $(i,j)$ with $j\leq j_i-i$ but otherwise
can be arbitrary. The $k$-set $\{j_1,\dots,j_k\}$ will be named
\emph{pivot set} of $X$. The positions of the 
unrestricted entries in $\mat{B}$ form the Ferrers diagram of an
integer partition, as shown in the following for the cases $v=8$, $k=3$,
$(i_1,i_2,i_3)=(1,2,3)$, $(2,4,7)$.
\begin{equation}
  \label{eq:etzion-ferrers}
  \ytableausetup{centertableaux,boxsize=1em}
  \setlength{\arraycolsep}{1em}
  \begin{array}{crl}
    \text{matrix shape}&
    \multicolumn{1}{c}{\text{Ferrers diagram}}&
    \multicolumn{1}{c}{\text{integer partition}}\\[1ex]
    \setlength{\arraycolsep}{.33em}
    \begin{pmatrix}
      1&0&0&*&*&*&*&*\\
      0&1&0&*&*&*&*&*\\
      0&0&1&*&*&*&*&*
    \end{pmatrix}&\ydiagram[\bullet]{5,5,5}&15=5+5+5\\[4ex]
    \setlength{\arraycolsep}{.33em}
    \begin{pmatrix}
      0&1&*&0&*&*&0&*\\
      0&0&0&1&*&*&0&*\\
      0&0&0&0&0&0&1&*
    \end{pmatrix}&\reflectbox{\ydiagram[\bullet]{4,3,1}}&8=4+3+1
  \end{array}
\end{equation}
The following lemma 
is a special case of \cite[Lemma
2]{etzion2009error}.
\begin{lemma}
  \label{lma:disjoint}
  If subspaces $X$, $Y$ of $\F_q^v$ have disjoint pivot sets, they are
  itself disjoint (i.e., $X\cap Y=\{\vek{0}\}$).
\end{lemma}
\begin{proof}
  A nonzero vector in $X$ must have its pivot (first nonzero position)
  in the pivot set of $X$, and similarly for $Y$. The result follows.
\qed\end{proof}
Now we focus on the special case in which the pivot set is
$\{1,\dots,k\}$, i.e., the canonical matrix has the ``systematic'' form
$\mat{A}=(\imat_k|\mat{B})$. For matrices
$\mat{A},\mat{B}\in\mathbb{F}_q^{k\times v}$ the \emph{rank distance}
is defined as
$\rdist(\mat{A},\mat{B}):=\operatorname{rk}(\mat{A}-\mat{B})$.  The
subspace distance of two $k$-subspaces with pivot set $\{1,\dots,k\}$
can be computed from the rank distance of the corresponding canonical
matrices:

\begin{lemma}[{\cite[Prop.~4]{silva2008rank}}]
  \label{lma:lifting}
  Let $X,X'$ be $k$-subspaces of $\F_q^v$ with canonical matrices
  $(\imat_k|\mat{B})$ and $(\imat_k|\mat{B}')$, respectively. Then
  $\sdist(X,X')=2\cdot\rdist(\mat{B},\mat{B}')$.  
  \footnote{More
    generally, this formula holds if $X$ and $X'$ have the same pivot
    set and $\mat{B},\mat{B}'\in\F_q^{k\times(v-k)}$ denote the
    corresponding complementary submatrices in their canonical matrices;
    see e.g.~\cite[Corollary 3]{MR2801585}.}
\end{lemma}
\begin{proof}
  The matrix $\left(
    \begin{smallmatrix}
      \imat_k&\mat{B}\\\imat_k&\mat{B}'
    \end{smallmatrix}\right)$ generates $X+X'$ and reduces via
  Gaussian elimination to $\left(
    \begin{smallmatrix}
      \imat_k&\mat{B}\\\mat{0}&\mat{B}'-\mat{B}
    \end{smallmatrix}\right)$. Hence
  $\dim(X+X')=k+\rk(\mat{B}'-\mat{B})=k+\rdist(\mat{B},\mat{B}')$ and
  $\sdist(X,X')=2\dim(X+X')-2k=2\,\rdist(\mat{B},\mat{B}')$.
\qed\end{proof}
The so-called \emph{lifting construction}
\cite[Sect.~IV.A]{silva2008rank} associates with a matrix code
$\mathcal{B}\subseteq\F_q^{k\times(v-k)}$ the constant-dimension code
$\mathcal{C}$ in $\F_q^v$ whose codewords are the $k$-spaces generated by
$(\imat_k|\mat{B})$, $\mat{B}\in\mathcal{B}$. By
Lemma~\ref{lma:lifting}, the code $\mathcal{C}$ is isometric to
$\mathcal{B}$ with scale factor $2$. In particular, $\mathcal{C}$ is
a partial $k$-spread if and only if $\mathcal{B}$ has minimum rank
distance $\rdist(\mathcal{B})=k$.

\begin{lemma}
  \label{lma:mrd}
  There exists a partial $k$-spread $\mathcal{S}$ of size $q^{v-k}$ in $\F_q^v$
  whose codewords cover precisely the points outside the
  ($v-k$)-subspace $S=\{\vek{x}\in\F_q^v;x_1=x_2=\dots=x_k=0\}$.
\end{lemma}
\begin{proof}
  Write $n=v-k$ and consider a matrix representation
  $M\colon\F_q^{n\times n}$ of $\F_{q^n}/\F_q$, obtained by
  expressing the multiplication maps
  $\mu_\alpha\colon\F_{q^n}\to\F_{q^n}$, $x\mapsto\alpha x$ (which are
  linear over $\F_q$) in terms of a fixed basis of
  $\F_{q^n}/\F_q$. Then $M(\alpha+\beta)=M(\alpha)+M(\beta)$,
  $M(\alpha\beta)=M(\alpha)M(\beta)$, $M(1)=\imat_n$, and hence all
  matrices in $M(\F_{q^n})$ are invertible and have mutual rank
  distance $n$.\footnote{In ring-theoretic terms, the matrices in
    $M(\F_{q^n})$ form a maximal subfield of the ring of $n\times n$ matrices
    over $\F_q$.}

  Now let $\mathcal{B}\subseteq\F_q^{k\times n}$ be the matrix code
  obtained from $M(\F_{q^n})$ by deleting the last $n-k$ rows, say, of
  every matrix. Then $\card{\mathcal{B}}=q^n$ and
  $\rdist(\mathcal{B})=k$. Hence by applying the lifting construction to
  $\mathcal{B}$ we obtain a partial $k$-spread $\mathcal{S}$ in
  $\F_q^v$ of size $q^n=q^{v-k}$ (Lemma~\ref{lma:lifting}). 

  The codewords in $\mathcal{S}$ cover only points outside $S$
  (compare the proof of Lemma~\ref{lma:disjoint}). It remains to show
  that every such point is covered. This can be done by a counting
  argument or in the following more direct fashion: Let
  $\vek{a}\in\F_q^k\setminus\{\vek{0}\}$, $\vek{b}\in\F_q^{v-k}$ be
  arbitrary vectors and consider the equation $\vek{a}\mat{X}=\vek{b}$ for
  $\mat{X}\in\mathcal{B}$. Since
  $\rk(\mat{X}-\mat{X'})=k$ for $\mat{X}\neq\mat{X}'$, the $q^{v-k}$
  elements $\vek{a}\mat{X}$, $\mat{X}\in\mathcal{B}$, are distinct and
  hence account for all elements in $\F_q^{v-k}$. Thus the equation
  has a solution $\mat{B}\in\mathcal{B}$, and the point
  $P=\F_q(\vek{a}|\vek{b})=\F_q\vek{a}(\imat_k|\mat{B})$ is covered
  by the codeword in $\mathcal{S}$ with canonical matrix
  $(\imat_k|\mat{B})$.
\qed\end{proof}
Now we are ready for the promised construction of large partial
$k$-spreads.
\begin{theorem}[\cite{beutelspacher1975partial}]
  \label{thm:multicomponent}
  Let $v,k$ be positive integers satisfying $v=tk+r$, $t\geq 2$ and
  $1\leq r\leq k-1$. There exists a partial $k$-spread $\mathcal{S}$
  in $\F_q^v$ of size
  \begin{equation*}
    \card{\mathcal{S}}=1+\sum_{s=1}^{t-1}q^{v-sk}=1+\sum_{s=1}^{t-1}q^{sk+r},
  \end{equation*}
  and hence we have
  $\smax_q(v,2k;k)\geq1+\sum_{s=1}^{t-1}q^{sk+r}$.
\end{theorem}
The corresponding bound for the deficiency is $\deficiency\leq q^r-1$. It
depends only on $k$ and the residue $r=v\bmod k$. 

It had been
conjectured in \cite[Sect.~2.2]{eisfeldt} that $\deficiency=q^r-1$ in
general, but this conjecture was later disproved in \cite{spreadsk3}
by exhibiting a maximal partial plane spread of size $34$ in $\F_2^8$,
which has deficiency $2^2-2$.

\begin{proof}
  The proof is by induction on $t$, using Lemma~\ref{lma:lifting} and
  applying the inductive hypothesis to $S\cong\F_q^{v-k}$. The case
  $v=k+r$, in which $\smax_q(v,2k;k)=1$, serves as the anchor of the
  induction.
\qed\end{proof}
The partial spread $\mathcal{S}$ exhibited in the proof of
Theorem~\ref{thm:multicomponent} consists of $t-1$ ``layers''
$\mathcal{S}_1$, \ldots, $\mathcal{S}_{t-1}$ of decreasing sizes
$\card{\mathcal{S}_s}=q^{v-sk}$, whose codewords are obtained from matrix
representations of $\F_{q^{v-sk}}$ and have their pivots in
positions $(s-1)k+1$, $(s-1)k+2$, \ldots, $sk$ (hence vanish on the
first $(s-1)k$ coordinates). The union $\bigcup_{s=1}^{t-1}\mathcal{S}_s$ leaves
exactly the points of a ($k+r$)-subspace $S$ of $\F_q^v$ (the span of the
last $k+r$ standard unit vectors) uncovered. Finally, one
further $k$-subspace $S_0$ of $S$ is selected to form
$\mathcal{S}=\bigcup_{s=1}^{t-1}\mathcal{S}_s\cup\{S_0\}$.\footnote{The
  space $S_0$ has been named \emph{moving subspace} of $\mathcal{S}$,
  since it can be freely ``moved'' within $S$ without affecting the
  partial spread property of $\mathcal{S}$.}

\begin{example}
  \label{ex:q=2,v=5,k=3}
  We consider the particular case $q=2$, $v=5$, $k=3$, in which
  $\card{\mathcal{S}}=2^3+1=9$. In this case there is only one layer
  $\mathcal{S}_1$, which can be obtained from a matrix representation
  of $\F_{8}$ as follows: Representing $\F_{8}$ as $\F_2(\alpha)$
  with $\alpha^3+\alpha+1=0$, we first express the powers $\alpha^j$,
  $0\leq j\leq 6$ in terms of the basis
  $1,\alpha,\alpha^2$ of $\F_{8}/\F_2$, as in the
  following matrix:
  \begin{equation}
    \label{eq:powers}
    \mat{M}=
    \begin{array}{r||ccccccc|cc}
      &\alpha^0&\alpha^1&\alpha^2&\alpha^3&\alpha^4&\alpha^5&\alpha^6&
      \alpha^0&\alpha^1\\\hline\hline
      \alpha^0&1&0&0&1&0&1&1&1&0\\
      \alpha^1&0&1&0&1&1&1&0&0&1\\
      \alpha^2&0&0&1&0&1&1&1&0&0
    \end{array}
  \end{equation}
  The seven consecutive $3\times 3$ submatrices of this matrix, which
  has been extended to the right in order to mimic the cyclic
  wrap-around, form a matrix field isomorphic to $\F_8$ (with
  $\mat{0}\in\F_2^{3\times 3}$ added). Similarly, the code
  $\mathcal{B}\subset\F_2^{2\times 3}$ is obtained by extracting the
  first seven consecutive $3\times 2$ submatrices, adding
  $\mat{0}\in\F_2^{3\times 2}$, and transposing; cf.\ the proof of
  Lemma~\ref{lma:mrd}). Prepending the $2\times 2$ identity matrix
  then gives the canonical matrices of the $8$ codewords of
  $\mathcal{S}_1$\,:
   \begin{gather*}
     \left(
       \begin{array}{cc|ccc}
         1&0&0&0&0\\
         0&1&0&0&0
       \end{array}\right),\;
     \left(
       \begin{array}{cc|ccc}
         1&0&1&0&0\\
         0&1&0&1&0
       \end{array}\right),\;
    \left(
       \begin{array}{cc|ccc}
         1&0&0&1&0\\
         0&1&0&0&1
       \end{array}\right),\;
     \left(
       \begin{array}{cc|ccc}
         1&0&0&0&1\\
         0&1&1&1&0
       \end{array}\right),\\
     \left(
       \begin{array}{cc|ccc}
         1&0&1&1&0\\
         0&1&0&1&1
       \end{array}\right),\;
     \left(
       \begin{array}{cc|ccc}
         1&0&0&1&1\\
         0&1&1&1&1
       \end{array}\right),\;
    \left(
       \begin{array}{cc|ccc}
         1&0&1&1&1\\
         0&1&1&0&1
       \end{array}\right),\;
     \left(
       \begin{array}{cc|ccc}
         1&0&1&0&1\\
         0&1&1&0&0
       \end{array}\right).
   \end{gather*}
   Finally, from the seven
   lines in the plane $S=\{\vek{x}\in\F_2^5;x_1=x_2=0\}$ a 9th
   codeword $L_0$ (moving line) is selected to form
   $\mathcal{S}=\mathcal{S}_1\cup\{L_0\}$.

   The partial line spread $\mathcal{S}$ is in fact maximal, as we
   will see in a moment, and represents one of the $4$ isomorphism
   types of maximal partial line spreads in
   $\PG(\F_2^5)=\PG(4,\F_2)$.\footnote{The full classification,
     including also partial line spreads of smaller size, can be found
     in \cite{Gordon-Shaw-Soicher-2004-unpublished}. To our best
     knowledge there is only one further nontrivial parameter case, where a
     classification of maximal (proper) partial spreads is known, viz.\
     the case of plane spreads in $\PG(6,\F_2)$, settled in
     \cite{hkk_partial_plane_spreads}.}
\end{example}

Now we will reduce the upper bound \eqref{eq:1st_ub} by a summand of
$q-1$, which is sufficient to settle the case $r=1$ and hence
determine the numbers $\smax_q(tk+1,2k;k)$. The key ingredient will be
the observation that a partial $k$-spread induces in every hyperplane
a vector space partition, whose members have dimension $k$, $k-1$, or
$1$. Before turning to the general case, which is a little technical,
we continue the preceding example and illustrate the method for
partial line spreads in $\F_2^5$.

The geometry $\PG(4,\F_2)$ has $31$ points, each line containing $3$
points, and thus it is conceivable that a partial line spread
$\mathcal{S}$ of size
$10$ exists in $\PG(4,\F_2)$. But in fact it does not. To prove this,
consider a hyperplane (solid) $H$ in $\PG(4,\F_2)$. If $H$ contains
$\alpha$ lines of $\mathcal{S}$, it meets the remaining
$\card{\mathcal{S}}-\alpha$ lines in a point, giving the constraint $\alpha\cdot
3+(\card{\mathcal{S}}-\alpha)\cdot 1\leq 15$, the total number of points in
$H$. This is equivalent to $\card{\mathcal{S}}\leq 15-2\alpha$. In order to
complete the proof, we need to show that there exists a hyperplane
containing at least $3$ lines of $\mathcal{S}$. This can be done by
an averaging argument. On average, a hyperplane contains
\begin{equation}
  \label{eq:Haverage}
  \sum_H \frac{\card{\{L\in\mathcal{S};L\subset H\}}}{2^5-1}
  =\sum_{L\in\mathcal{S}} \frac{\card{\{H;H\supset L\}}}{2^5-1}
  =\frac{\card{\mathcal{S}}(2^3-1)}{2^5-1}=\frac{7}{31}\cdot\card{\mathcal{S}}
\end{equation}
lines of $\mathcal{S}$. If $\card{\mathcal{S}}\geq 9$, this number is
$>2$, implying the desired conclusion $\deficiency\geq 1$.\footnote{In
  this particular case one may also argue as follows: If
  $\card{\mathcal{S}}=10$ then there is only one hole and the
  hyperplane constraint becomes $3\alpha+(10-\alpha)+h=15$, where
  $h\in\{0,1\}$. This forces $\alpha=2$ and $h=1$, i.e., every
  hyperplane should contain the hole. This is absurd, of course.}



The general case is the subject of the following

\begin{theorem}[{\cite[Th.~2.7(a)]{eisfeldt}}]
  \label{thm:lb_def}
  The deficiency of a maximal $k$-spread in $\F_q^v$, where $k$ does not divide $v$, is at least $q-1$.
\end{theorem}
\begin{proof}
  Reasoning as in the preceding example gives
  $\alpha\cdot\frac{q^k-1}{q-1}+(\card{\mathcal{S}}-\alpha)\frac{q^{k-1}-1}{q-1}
  \leq\frac{q^{tk+r-1}-1}{q-1}$ and hence the bound
  \begin{equation}
    \label{eq:lb_def_p1}
    \card{\mathcal{S}}\leq\frac{q^{tk+r-1}-1
      -\alpha(q^k-q^{k-1})}{q^{k-1}-1}
  \end{equation}
  for any partial $k$-spread $\mathcal{S}$ having a hyperplane incident
  with $\alpha$ members of $\mathcal{S}$.
  
  Now suppose $\card{\mathcal{S}}=1+\sum_{s=1}^{t-1}q^{sk+r}$, the same
  size as the partial $k$-spread in Theorem~\ref{thm:multicomponent}. In this
  case the average number of codewords contained in a hyperplane is
  \begin{align*}
    \frac{q^{(t-1)k+r}-1}{q^{tk+r}-1}\left(1+\sum_{s=1}^{t-1}q^{sk+r}\right)
    &=\frac{1}{q^{tk+r}-1}
    \left(\sum_{s=t}^{2t-2}q^{sk+2r}-\sum_{s=1}^{t-2}q^{sk+r}-1\right)\\
    &=\frac{1}{q^{tk+r}-1}
    \left(\sum_{s=t}^{2t-2}q^{sk+2r}-\sum_{s=0}^{t-2}q^{sk+r}+q^r-1\right)\\
    &=\sum_{s=0}^{t-2}q^{sk+r}+\frac{q^r-1}{q^{tk+r}-1}.
  \end{align*}
  It follows that $\mathcal{S}$, and likewise all partial $k$-spreads
  of deficiency $\leq q^r-1$, have a hyperplane containing at least
  $1+\sum_{s=0}^{t-2}q^{sk+r}$ codewords. Substituting this number
  into \eqref{eq:lb_def_p1} gives
  \begin{align*}
    \card{\mathcal{S}}&\leq\frac{1}{q^{k-1}-1}\left(q^{tk+r-1}-1
                   -\biggl(1+\sum_{s=0}^{t-2}q^{sk+r}\biggr)(q^k-q^{k-1})\right)\\
    &=\frac{1}{q^{k-1}-1}\left(q^{tk+r-1}-1-q^k+q^{k-1}
      -\sum_{s=1}^{t-1}q^{sk+r}+\sum_{s=0}^{t-2}q^{sk+r+k-1}\right)\\
    &=1+\sum_{s=1}^{t-2}q^{sk+r}+\frac{q^{tk+r-1}-q^k-q^{(t-1)k+r}+q^{r+k-1}}
      {q^{k-1}-1}\\
    &=1+\sum_{s=1}^{t-1}q^{sk+r}+\frac{q^{r+k-1}-q^k}{q^{k-1}-1}\\
    &=1+\sum_{s=1}^{t-1}q^{sk+r}+q^r-q+\frac{q^r-q}{q^{k-1}-1}\\
    &=\sum_{s=0}^{t-1}q^{sk+r}-(q-1)+\frac{q^r-q}{q^{k-1}-1},
  \end{align*}
  valid now for any partial $k$-spread $\mathcal{S}$ in $\F_q^v$. Since
  the last summand is $<1$, we obtain the desired conclusion
  $\deficiency\geq q-1$.
\qed\end{proof}

Theorem~\ref{thm:lb_def} has the following immediate corollary,
established by Beutelspacher in 1975, which settles the case $r=1$
completely.
%
\begin{corollary}[{\cite[Th.~4.1]{beutelspacher1975partial}}; see also
  \cite{hong1972general} for the special case $q=2$]
  \label{cor:r=1}
  For integers $k\geq 2$ and $v=tk+1$ with $t\geq 1$
  we have
  $\smax_q(v,2k;k)=1+\sum_{s=1}^{t-1}q^{sk+1}$,\footnote{\label{fn:written}
    This can also written as $\smax_q(v,2k;k)=
    q^1\cdot
    \frac{q^{v-1}-1}{q^k-1}-q+1=\frac{q^v-q^{k+1}+q^k-1}{q^k-1}$.}
  with corresponding deficiency $\deficiency=q-1$.\footnote{The
    corresponding number of holes is $q^k$.}
\end{corollary}
In particular, maximal partial line spreads in $\F_q^v$, $v$ odd (the
case where no line spreads exist), have
size $q^{v-2}+q^{v-4}+\dots+q^3+1$, deficiency $q-1$, and $q^2$ holes.

In his original proof of the corollary Beutelspacher considered the set
of holes $N$ and the average number of holes per hyperplane, which is
less than the total number of holes divided by $q$. An important
insight
was the relation $\card{N}\equiv\card{(H\cap N)}\pmod{q^{k-1}}$ for each
hyperplane $H$, i.e., the number of holes per hyperplane satisfies a
certain modulo constraint. We will see this concept in full generality
in Section~\ref{section_q_r_divisible}. In terms of integer linear
programming, the upper bound is obtained by an integer rounding
cut. The construction in \cite[Theorem 4.2]{beutelspacher1975partial}
recursively uses arbitrary $k'$-spreads, so that it is more general
than the one of Theorem~\ref{thm:multicomponent}.


  
For a long time the best known upper bound on $\smax_q(v,2k;k)$, i.e., the best
known lower bound on $\deficiency$, was the one obtained by Drake and
Freeman in 1979:
\begin{theorem}[Corollary 8 in \cite{nets_and_spreads}]
  \label{thm_partial_spread_4}
  The deficiency of a maximal partial $k$-spread in $\F_q^v$ 
  is at least $\lfloor\theta\rfloor+1=\lceil\theta\rceil$,\footnote{Assuming 
  $1+4q^k(q^k-q^r)=1+4q^{k+r}(q^{k-r}-1)=(2z-1)^2=1+4z(z-1)$ for some integer $z>1$ implies 
  $q^{k+r}\mid z$ or $q^{k+r}\mid z-1$, so that $z\ge q^{k+r}$, which is impossible for $(k,r)\neq (1,0)$. Thus, 
  $2\theta\notin \mathbb{Z}$, so that $\theta\notin\mathbb{Z}$ and $\lfloor\theta\rfloor+1=\lceil\theta\rceil$.}
  where $2\theta=\sqrt{1+4q^k(q^k-q^r)}-(2q^k-2q^r+1)$.
\end{theorem}
The authors concluded from the existence of a partial spread the
existence of a (group constructible) $(s,r,\mu)$-net and applied
\cite[Theorem~1B]{bose1952orthogonal}---a necessary existence
criterion formulated for orthogonal arrays of strength $2$ by Bose and
Bush in 1952. The underlying proof technique can be further traced
back to \cite{plackett1946design} and is strongly related to the
classical second-order Bonferroni Inequality
\cite{bonferroni1936teoria,galambos1977bonferroni}; 
see also 
\cite[Section 2.5]{honold2015constructions} for an application to
bounds for subspace codes. 
  
Given Theorem~\ref{thm_spread} and Corollary~\ref{cor:r=1}, the first
open binary case is $\smax_2(8,6;3)$.  The construction from
Theorem~\ref{thm:multicomponent} gives a partial spread of cardinality
$33$, while Theorem~\ref{thm_partial_spread_4} implies an upper bound
of $34$. As already mentioned, in 2010 El-Zanati et
al. \cite{spreadsk3} found a sporadic partial plane spread in $\F_2^8$
of cardinality $34$ by a computer search. Together with the following
easy lemma, this completely answers the situation for partial plane
spreads in $\F_2^v$; see Corollary~\ref{cor_partial_spread_k3} below.

\begin{lemma}
  \label{lemma_extend_construction}
  For fixed $q$, $k$ and $r$ the deficiency $\deficiency$ is a
  non-increasing function of $v=kt+r$.  
\end{lemma}
\begin{proof}
  Let $\mathcal{S}$ be a maximal partial $k$-spread in $\F_q^{tk+r}$
  and $\deficiency$ its deficiency, so that
  $\smax_q(tk+r,2k;k)=\sum_{s=0}^{t-1}q^{sk+r}-\deficiency$.  We can
  embed $\mathcal{S}$ into $\F_q^{(t+1)k+r}$ by prepending $k$ zeros
  to each codeword. Then Lemma~\ref{lma:mrd} can be applied and yields
  a partial $k$-spread $\mathcal{S}'$ in $\F_q^{(t+1)k+r}$ of size
  $q^{tk+r}$, whose codewords are disjoint from those in $\mathcal{S}$.
  This implies
  $\smax_q((t+1)k+r,2k;k)\geq\card{\mathcal{S}\cup\mathcal{S}'}
  =\sum_{s=0}^{t}q^{sk+r}-\deficiency$, and hence the deficiency
  $\deficiency'$ of a maximal partial $k$-spread in $\F_q^{(t+1)k+r}$
  satisfies $\deficiency'\leq\deficiency$.
\qed\end{proof}

So, any improvement of the best known lower bound for a single
parameter case gives rise to an infinite series of improved lower
bounds.  Unfortunately, so far, the sporadic construction in
\cite{spreadsk3} is the only known example being strictly superior to
the general construction of Theorem~\ref{thm:multicomponent}.

\begin{corollary}
  \label{cor_partial_spread_k3}
  For each integer $m\ge 2$ we have
  $\smax_2(3m,6;3)=\frac{2^{3m}-1}{7}$,
  $\smax_2(3m+1,6;3)=\frac{2^{3m+1}-9}{7}$, and
  $\smax_2(3m+2,6;3)=\frac{2^{3m+2}-18}{7}$. The corresponding
    deficiencies are $0$, $1$ and $2$, respectively.
\end{corollary}

Very recently, the case $q=r=2$ was completely settled. For $k=3$ the answer
is given in the preceding corollary, and for $k\geq 4$ by the following
\begin{theorem}[{\cite[Theorem 5]{kurzspreads}}]
  \label{thm_spread_exact_value_3}
  For integers $k\ge 4$ and $v=tk+2$ with $t\geq 1$ we have
  $\deficiency=3$ and 
  $\smax_2(kt+2,2k;k)=1+\sum_{s=1}^{t-1}2^{tk+2}=\frac{2^{kt+2}-3\cdot
    2^{k}-1}{2^k-1}$.\footnote{Thus in all these cases $\deficiency=2^2-1$
    and the partial spreads of Theorem~\ref{thm:multicomponent} are
    maximal. This notably differs from the case
    $k=3$.}
\end{theorem}

The technique used to prove this theorem is very similar to the one
presented in the proof of Theorem~\ref{thm:lb_def}.

\begin{corollary}
  \label{cor_spread_k_4}
  We have $\smax_2(4m,8;4)=\frac{2^{4m}-1}{15}$.
  $\smax_2(4m+1,8;4)=\frac{2^{4m+1}-17}{15}$,
  $\smax_2(4m+2,8;4)=\frac{2^{4m+2}-49}{15}$, and
  $\frac{2^{4m+3}-113}{15}\le
  \smax_2(4m+3,8;4)\le\frac{2^{4m+3}-53}{15}$ for all $m\ge 2$. The
  corresponding deficiencies are $0$, $1$, $3$ and
  $3\leq\deficiency\leq 7$, respectively
\end{corollary}

As a consequence, the first unknown binary case is now
$129\le\smax_2(11,8;4)\le 133$.\footnote{The upper bound can be
  sharpened to $132$, as we will see later.} For $r=2$ and $q=3$ the
upper bound of Theorem~\ref{thm_partial_spread_4} has been decreased
by $1$:

\begin{lemma}[cf.\ {\cite[Lemma 4]{kurzspreads}}]
  \label{lemma_spread_upper_bound_3_q}
  For integers $t\ge 2$ and $k\ge 4$, we have $\deficiency\geq 5$ and 
  $\smax_3(kt+2,2k;k) \le
  \frac{3^{kt+2}-3^2}{3^k-1}-5$.
\end{lemma}
Again, the proof technique is very similar to that used in the proof of
Theorem~\ref{thm:lb_def}.

Theorem~\ref{thm:multicomponent} is asymptotically optimal for
$k\gg r=v\bmod k$, as recently shown by N{\u{a}}stase and Sissokho:
\begin{theorem}[{\cite[Theorem 5]{nastase2016maximum}}]
  \label{thm_partial_spread_asymptotic}
  If $k>\gauss{r}{1}{q}$
  then $\deficiency=q^r-1$ and
  $\smax_q(v,2k;k)=1+\sum_{s=1}^{t-1}q^{sk+r}$.\footnote{This
    corresponds again to the upper bound $\deficiency=q^r-1$.}
\end{theorem}
Choosing $q=r=2$, this result
covers Theorem~\ref{thm_spread_exact_value_3}. The
same authors have refined their analysis, additionally using
Theorem~\ref{thm_length_of_tail} from the theory of vector space
partitions, to obtain improved upper bounds for some of the cases
$k\le \gauss{r}{1}{q}$, see \cite[Theorem 6 and 7]{nastase2016maximumII}.
Using the theory of $q^r$-divisible codes, presented in the next section, 
we extend their results further in Corollary~\ref{cor_main_theorem_1} 
and Theorem~\ref{main_theorem_2}.

\section{$q^r$-divisible sets and codes}
\label{section_q_r_divisible}

The currently most effective approach to good upper bounds for partial
spreads follows the original idea of Beutelspacher and considers the
set of holes as a stand-alone object. As it appears in the proof of
Beutelspacher, the number of holes in a hyperplane satisfies a certain
modulo constraint. In this section we consider sets of points in
$\aspace$ having the property that modulo some integer $\Delta>1$
the number of points in each hyperplane is the same. Such point sets
are equivalent to $\Delta$-divisible codes
\cite{ward1999introduction,ward2001divisible_survey} with projectively
distinct coordinate functionals (so-called \emph{projective} codes), and this
additional restriction forces $\Delta$ to be a power of the same prime
as $q$. Writing $q=p^e$, $p$ prime, and $\Delta=p^f$, we have
$\Delta=q^r$ with $r=f/e\in\frac{1}{e}\mathbb{Z}$.

We will derive several important properties of these
\emph{$q^r$-divisible sets} and \emph{codes} and in particular
observe that the set of holes of a partial spread is exactly of this
type. Without the notion of $q^r$-divisible sets and the reference to
the linear programming method, almost all results of this section are
contained in \cite{kurzspreadsII}. A more extensive introduction to
the topic, including constructions and relations to other
combinatorial objects, is currently in preparation
\cite{dsmta:q-r-divisble}.

In what follows, we denote the point set
of $\aspace$ by $\points$ and call for subsets
$\mathcal{C}\subseteq\points$ and subspaces
$X$ of $\F_q^v$ the integer $\card{(\mathcal{C}\cap
X)}=\card{\{P\in\mathcal{C};P\subseteq X\}}$ the \emph{multiplicity} of $X$
  with respect to $\mathcal{C}$.

\begin{definition}
  \label{def_delta_divisible}
  Let $\Delta>1$ be an integer. A set $\mathcal{C}$ of points in
  $\aspace$ is called \emph{weakly $\Delta$-divisible} if there exists
  $u\in\mathbb{Z}$ with
  $\card{(\mathcal{C}\cap H)}\equiv u\pmod{\Delta}$ for each hyperplane
  $H$ of $\aspace$. If $u\equiv\card{\mathcal{C}}\pmod{\Delta}$, we
  call $\mathcal{C}$ \emph{(strongly) $\Delta$-divisible}.
\end{definition}
Trivial cases are $\mathcal{C}=\emptyset$ (strongly $\Delta$-divisible for any
$\Delta$) and $\mathcal{C}=\points$ (weakly $\Delta$-divisible for any
$\Delta$, with largest strong divisor $\Delta=q^{v-1}$).\footnote{This
  is also true for $v=1$, where $\mathcal{C}=\emptyset,\points$
  exhausts all possibilities.}

It is well-known (see, e.g., \cite[Prop.~1]{tsfasman-vladut95,dodunekov1998codes}) that
the relation $C\to\mathcal{C}$, associating with a full-length linear
$[n,v]$ code $C$ over $\F_q$ the $n$-multiset $\mathcal{C}$ of points
in $\aspace$ defined by the columns of any generator matrix, induces a
one-to-one correspondence between classes of \mbox{(semi-)linearly}
equivalent spanning multisets and classes of \mbox{(semi-)monomially}
equivalent full-length linear codes. Point
sets 
correspond in this way to projective linear codes, which are also
characterized by the condition $\hdist(C^\perp)\geq 3$. The importance
of the correspondence lies in the fact that it relates
coding-theoretic properties of $C$ to geometric or combinatorial
properties of $\mathcal{C}$. An example is the formula
\begin{equation}
  \label{eq:hweight-npoints}
  \hweight(\vek{a}\mat{G})=n-\card{\{1\leq j\leq
  n;\vek{a}\cdot\vek{g}_j=0\}}
  =n-\card{(\mathcal{C}\cap\vek{a}^\perp)},
\end{equation}
where $\hweight$ denotes the Hamming weight,
$\mat{G}=(\vek{g}_1|\dots|\vek{g}_n)\in\F_q^{v\times n}$ a generating
matrix of $C$, $\vek{a}\cdot\vek{b}=a_1b_1+\dots+a_vb_v$, and
$\vek{a}^\perp$ is the hyperplane in $\aspace$ with equation
$a_1x_1+\dots+a_vx_v=0$.

A linear code $C$ is said to be \emph{$\Delta$-divisible}
($\Delta\in\mathbb{Z}_{>1}$) if all nonzero codeword weights are
multiples of $\Delta$. Following the Gleason-Pierce-Ward Theorem on
the divisibility of self-dual codes (see, e.g.,
\cite[Ch.~9.1]{huffman-pless03}), a rich theory of divisible codes has
been developed over time, mostly by H.~N.~Ward; cf.\ his survey
\cite{ward2001divisible_survey}. One of Ward's results implies that
nontrivial weakly $\Delta$-divisible point sets in $\aspace$ are
strongly $\Delta$ divisible and exist only in the case
$\Delta=p^f$. The proof uses the so-called \emph{standard
  equations} for the hyperplane spectrum of $\mathcal{C}$, which we
state in the following lemma. The standard equations are equivalent to
the first three MacWilliams identities for the weight enumerators of
$C$ and $C^\perp$ (stated as Equation~\eqref{mac_williams_identies}
below), specialized to the case of projective linear codes. The
geometric formulation, however, seems more in line with the rest of
the paper.




\begin{lemma}
  \label{lemma_standard_equations_q}
  Let $\mathcal{C}$ be a set of points in $\aspace$ with
  $\card{\mathcal{C}}=n$, and let $a_i$ hyperplanes of $\aspace$
  contain exactly $i$ points of $\mathcal{C}$ ($0\le i\le n$). Then we
  have
  \begin{eqnarray}
    \sum_{i=0}^{n}a_i &=& \gauss{v}{1}{q},\label{eq_ste1}\\
    \sum_{i=1}^{n}ia_i &=& n\cdot \gauss{v-1}{1}{q},\label{eq_ste2}\\
    \sum_{i=2}^{n}{i\choose 2} a_i &=& {n\choose 2}\cdot \gauss{v-2}{1}{q}\label{eq_ste3}.\footnotemark
  \end{eqnarray}
\end{lemma}
\footnotetext{The general (multiset) version of \eqref{eq_ste3} has an
additional summand of $q^{v-2}\cdot\sum_{P\in\points}\binom{\mathcal{C}(P)}{2}$ 
on the right-hand side, accounting for the fact that ``pairs of equal points''
are contained in $\gauss{v-1}{1}{q}$ hyperplanes.}
\begin{proof}
  Double-count incidences of the tuples $(H)$, $(P_1,H)$, and
  $(\{P_1,P_2\},H)$, where $H$ is a hyperplane and $P_1\neq P_2$ are
  points contained in $H$.  
\qed\end{proof}
In the proof of Theorem~\ref{thm:divisible} we will need that
\eqref{eq_ste1} and \eqref{eq_ste2} remain true for any multiset
$\mathcal{C}$ of points in $\aspace$, provided points are counted with
their multiplicities in $\mathcal{C}$ and the cardinality
$\card{\mathcal{C}}$ is defined in the obvious way. We will also need the
following concept of a \emph{quotient multiset}. Let $\mathcal{C}$ be
a set of points in $\aspace$ and $X$ a
subspace of $\F_q^v$. Define the multiset $\mathcal{C}/X$ of points in
the quotient geometry $\PG(\F_q^v/X)$ by assigning to a point $Y/X$ of
$\PG(\F_q^v/X)$ (i.e., $Y$ satisfies $\dim(Y/X)=1$) the difference
$\card{(\mathcal{C}\cap Y)}-\card{(\mathcal{C}\cap
X)}=\card{(\mathcal{C}\cap Y\setminus X)}$ as multiplicity.\footnote{This
  definition can be extended to multisets $\mathcal{C}$ by defining
  the multiplicity of $Y/X$ in $\mathcal{C}/X$ as the sum of the
  multiplicities in $\mathcal{C}$ of all points in $Y\setminus X$.}
With this definition it is obvious that
$\card{(\mathcal{C}/X)}=\card{\mathcal{C}}-\card{(\mathcal{C}\cap X)}$.
In particular,
if $\mathcal{C}$ is an $n$-set and $X=P$ is a point then
$\card{(\mathcal{C}/P)}=n-1$ or $n$, according to whether $P\in\mathcal{C}$
or $P\notin\mathcal{C}$, respectively.\footnote{If
  $C\leftrightarrow\mathcal{C}$ then the multisets $\mathcal{C}/P$,
  $P\in\points$, are associated to the ($v-1$)-dimensional subcodes
  $D\subset C$, and the $n$ points $P\in\mathcal{C}$ correspond to
  the $n$ subcodes $D$ of effective length $n-1$ (``$D$ is $C$
  shortened at $P$''). This correspondence between points and subcodes  
  extends to a correlation between $\aspace$ and $\PG(C/\F_q)$, which
  includes the familiar correspondence between hyperplanes and codewords as a
  special case; see \cite{tsfasman-vladut95,dodunekov1998codes} for details.} 
\begin{theorem}
  \label{thm:divisible}
  Let $\mathcal{C}\neq\emptyset,\points$ be a weakly
  $\Delta$-divisible point set in $\aspace$, $n=\card{\mathcal{C}}$,
  and $C$ any linear $[n,v]$-code over $\F_q$ associated with
  $\mathcal{C}$ as described above.\footnote{It is not required that
    $\mathcal{C}$ is spanning; if it is not then (iii)
    sharpens to ``$\Delta$ is a divisor of
    $q^{\dim\langle\mathcal{C}\rangle-2}$''.} Then
  \begin{enumerate}[(i)]
  \item $\mathcal{C}$ is strongly $\Delta$-divisible;
  \item $C$ is $\Delta$-divisible;
  \item $\Delta$ is a divisor of $q^{v-2}$.
  \end{enumerate}
\end{theorem}
\begin{proof}
  (i) and (ii) are equivalent in view of
  \eqref{eq:hweight-npoints}. 
  
  First we prove (i). Let $u$ be as in
  Definition~\ref{def_delta_divisible}. Choose a point
  $P\notin\mathcal{C}$ and let $\mathcal{C}'=\mathcal{C}/P$. Then
  $\mathcal{C}'$ is an $n$-multiset of points in
  $\PG(\F_q^v/P)\cong\PG(v-2,\F_q)$ with
  $\card{(\mathcal{C}'\cap H')}\equiv u\pmod{\Delta}$ for each
  hyperplane $H'$ of $\PG(\F_q^v/P)$. The hyperplane spectrum $(a_i')$
  of $\mathcal{C}'$ satisfies \eqref{eq_ste2} with $v$ replaced by
  $v-1$. Multiplying the identity for $\mathcal{C}'$ by $q$ and
  subtracting from it the identity for $\mathcal{C}$ gives
  \begin{equation*}
    \sum_{i\geq
      0}(u+i\Delta)(a_{u+i\Delta}-qa_{u+i\Delta}')=n\left(\frac{q^{v-1}-1}{q-1}
    -\frac{q^{v-1}-q}{q-1}\right)=n.
  \end{equation*}
Reading this equation modulo $\Delta$ and using \eqref{eq_ste1}
further gives
\begin{equation*}
  n\equiv u \sum_{i\geq
      0}(a_{u+i\Delta}-qa_{u+i\Delta}')=u\left(\frac{q^{v}-1}{q-1}
    -\frac{q^{v}-q}{q-1}\right)=u\pmod{\Delta},
\end{equation*}
as desired. Thus (i) and (ii) hold.

For the proof of (iii) we use a point $Q\in\mathcal{C}$ and its
associated quotient multiset
$\mathcal{C}''=\mathcal{C}/Q$, which satisfies $\card{\mathcal{C}''}=n-1$ and
$\card{(\mathcal{C}''\cap
H'')}\equiv u-1\pmod{\Delta}$ for each hyperplane
  $H''$ of $\PG(\F_q^v/Q)$. Subtracting \eqref{eq_ste2}
  for $\mathcal{C}'$ and $\mathcal{C}''$ gives
\begin{equation*}
    \sum_{i\geq
      0}(u+i\Delta)a_{u+i\Delta}'-\sum_{i\geq
      0}(u-1+i\Delta)a_{u-1+i\Delta}''
    =\frac{q^{v-2}-1}{q-1}.
  \end{equation*}
  Again reading the equation modulo $\Delta$ and using \eqref{eq_ste1}
  gives
  \begin{equation*}
    \frac{q^{v-2}-1}{q-1}\equiv u \sum_{i\geq
      0}a_{u+i\Delta}'-(u-1)\sum_{i\geq
      0}a_{u-1+i\Delta}''
    =\frac{q^{v-1}-1}{q-1}\pmod{\Delta}
  \end{equation*}
  or $q^{v-2}\equiv0\pmod{\Delta}$, as asserted.
\qed\end{proof}
Let us remark that Part~(ii) of Theorem~\ref{thm:divisible} also
follows from \cite[Th.~3]{ward1999introduction}, which asserts that a
not necessarily projective code $C$ satisfying the assumption of the
theorem must be either $\Delta$-divisible or the juxtaposition of a
$\Delta$-divisible code and a $v$-dimensional linear constant weight
code. Since the latter is necessarily a repetition of simplex codes,
this case does not occur for projective codes. Our proofs of (i),
(iii) use the very same ideas as in \cite{ward1999introduction},
translated into the geometric framework.\footnote{Readers may have
  noticed that, curiously, the 3rd standard equation (which
  characterizes projective codes) was not used at all in the proof.}

Part~(iii) of Theorem~\ref{thm:divisible} says that exactly
$\Delta$-divisible point sets in $\aspace$ exist only if $\Delta=q^r$
with $r\in\frac{1}{e}\mathbb{Z}$ and $r\leq v-1$;\footnote{By this we
  mean that $\Delta$ is the largest divisor in the sense of
  Definition~\ref{def_delta_divisible} or
  Theorem~\ref{thm:divisible}.} the whole point set $\points$ has
$\Delta=q^{v-1}$, and $v-2<r<v-1$ does not occur. Conversely, it is
not difficult to see that every divisor $\Delta>1$ of $q^{v-2}$ is the
largest divisor of some point set in $\aspace$.\footnote{If
  $t=\lfloor r\rfloor$ and $r'\in\{0,1,\dots,e-1\}$ is defined by $r=t+
  r'/e$, then the union of $p^{r'}$ parallel affine
  subspaces of dimension $t+1$ has this property.}

In the proof of Theorem~\ref{thm:divisible} we have used that the
(weak) divisibility properties of $\mathcal{C}$ and its quotient
multisets $\mathcal{C}/X$ are the same. Now we consider the
restrictions $\mathcal{C}\cap X$, which correspond to residual codes
of the associated code $C$.
\begin{lemma}
  \label{lemma_heritable}
  Suppose that $\mathcal{C}$ is a $q^r$-divisible set of points in
  $\aspace$ and $X$ a ($v-j$)-subspace of $\F_q^v$ with $1\leq j<r$. Then the
  restriction $\mathcal{C}\cap X$ is $q^{r-j}$-divisible.
\end{lemma}
\begin{proof}
  By induction, it suffices to consider the case $j=1$, i.e., $X=H$ is
  a hyperplane in $\aspace$.

  The hyperplanes of $\PG(H)$ are the ($v-2$)-subspaces of $\F_q^v$
  contained in $H$. Hence the assertion is equivalent to
  $\card{\mathcal{C}\cap U}\equiv\card{\mathcal{C}}=u\pmod{q^{r-1}}$ for every
  ($v-2$)-subspace $U\subset\F_q^v$. By assumption we have
  $\card{(\mathcal{C}\cap H_i)}\equiv u\pmod{q^r}$ for
  the $q+1$ hyperplanes $H_1,\dots,H_{q+1}$ lying above $U$. This gives
  \begin{equation*}
    (q+1)u\equiv\sum_{i=1}^{q+1}\card{(\mathcal{C}\cap H_i)}
    =q\cdot\card{(\mathcal{C}\cap
      U)}+\card{\mathcal{C}}\equiv q\cdot\card{(\mathcal{C}\cap
      U)}+u\pmod{q^r}
  \end{equation*}
  and hence $u\equiv\card{(\mathcal{C}\cap
    U)}\pmod{q^{r-1}}$, as claimed.
\qed\end{proof}


As mentioned at the beginning of this section, the set of holes of a
partial spread provides an example of a $q^r$-divisible set. The
precise statement is given in the following theorem,
which is formulated for general vector space partitions.

\begin{theorem}
  \label{thm:psp-vsp}
  \begin{enumerate}[(i)]
  \item Let $\mathcal{C}$ be a vector space partition of $\F_q^v$ of
    type $t^{m_t}\dotsm s^{m_s}1^{m_1}$ with $m_s>0$ (i.e.,
    $\mathcal{C}$ has a member of dimension $>1$ and $s$ chosen as
    the smallest such dimension). Then the points (i.e.,
    $1$-subspaces) in $\mathcal{C}$ form a $q^{s-1}$-divisible set.
  \item The holes of a partial $k$-spread in $\F_q^v$ form a
    $q^{k-1}$-divisible set.
  \end{enumerate}
\end{theorem}
\begin{proof}
  It is immediate that (i) implies (ii). For the proof of (i) let $H$
  be a hyperplane of $\aspace$. The points in $\points\setminus H$ are
  partitioned into the affine subspaces $X\setminus H$ for those
  $X\in\mathcal{C}$ satisfying $X\nsubseteq H$. If such a $t$-subspace
  $X$ is not a point,
  we have $t\geq s$ and hence $\card{(X\setminus H)}=q^{t-1}\equiv
  0\pmod{q^{s-1}}$. Moreover, we also have $\card{(\points\setminus
  H)}=q^{v-1}\equiv 0\pmod{q^{s-1}}$. It follows that the number of
  points in $\mathcal{C}$ that are not contained in $H$ is divisible
  by 
  $q^{s-1}$ 
  as well, completing the proof.
\qed\end{proof}

Theorem~\ref{thm:psp-vsp} explains our motivation for studying
$q^r$-divisible points sets in $\aspace$. Before delving deeper into
this topic, we pause for a few example applications to partial spreads,
which may help advertising our approach.

First we consider the problem of improving the upper bound
\eqref{eq:1st_ub} for the size of a partial $k$-spread in
$\aspace$. The bound is equivalent to
$\card{\mathcal{C}}\geq\frac{q^r-1}{q-1}$ for the corresponding hole
sets, which are $q^{k-1}$-divisible by
Theorem~\ref{thm:psp-vsp}(ii). But the smallest nontrivial
$q^{k-1}$-divisible point sets in $\aspace$ are the $k$-subspaces of
$\F_q^v$, since these are associated to the constant-weight-$q^{k-1}$
simplex code.\footnote{This follows, e.g., by applying the Griesmer
  bound to the associated linear code, which has minimum
  distance $\geq q^{k-1}$ and dimension $\geq k$.}  
Thus $\card{\mathcal{C}}\geq\frac{q^k-1}{q-1}>\frac{q^r-1}{q-1}$, and
equality in \eqref{eq:1st_ub} is not possible. Together with
Theorem~\ref{thm:multicomponent} this already gives the numbers
$\smax_2(tk+1,2k;k)$.

The preceding argument gives $\smax_2(8,6;3)\leq 35$, 
and as a second application, we now exclude the existence of a partial plane
spread of size $35$ in $\F_2^8$. As already mentioned, this also
follows from the Drake-Freeman bound
(Theorem~\ref{thm_partial_spread_4}) and forms an important ingredient
in the determination of the numbers $\smax_2(v,6;3)$. The
hole set $\mathcal{C}$ of such a partial plane spread has size
$2^8-1-35\cdot 7=10$ and is $4$-divisible, i.e., it meets every
hyperplane in $2$ or $6$ points. 

We claim that
$\dim\langle\mathcal{C}\rangle=4$. The inequality
$\dim\langle\mathcal{C}\rangle\geq 4$ is immediate from
$\card{\mathcal{C}}=10$. The reverse inequality follows from the fact
that the linear code $C$ associated with $\mathcal{C}$
is doubly-even, hence self-orthogonal, but cannot be
self-dual.\footnote{Just recall that the length of any doubly-even self-dual
binary code must be a multiple of $8$.}

Given that $\dim\langle\mathcal{C}\rangle=4$, the existence of
$\mathcal{C}$ is readily excluded using the standard equations:
\begin{equation}
  \label{eq:ste_ex2}
  \begin{array}{rrrrl}
    a_2&+&a_6&=&15,\\
    2a_2&+&6a_6&=&10\cdot 7,\\
    \binom{2}{2}a_2&+&\binom{6}{2}a_6&=&\binom{10}{2}\cdot 3.
  \end{array}
\end{equation}
The unique solution of the first two equations is $a_2=5$, $a_6=10$
(corresponding to the $2$-fold repetition of the $[5,4,2]$ even-weight
code), but it does not satisfy the third equation (since this code is
not projective).\footnote{Adding $20=5\cdot 4$, which accounts for the $5$
  pairs of equal points in the code, to the right-hand side
  ``corrects'' the third equation.}
 

As already mentioned, the standard equations in
Lemma~\ref{lemma_standard_equations_q} have a natural generalization
in the language of linear codes. 
To this end let $\mathcal{C}$ be a point set of size
$\card{\mathcal{C}}=n$ in $\PG(k-1,\F_q)$,
which is spanning\footnote{This assumption is necessary for the
relation $A_i=(q-1)a_{n-i}$ to hold.}, and $C$
the corresponding projective linear $[n,k]$ code over $\mathbb{F}_q$. 
The hyperplane spectrum $(a_i)_{0\leq i\leq n}$ of $\mathcal{C}$ and
the weight distribution $(A_i)_{0\leq i\leq n}$ of $C$ are related by
$A_i=(q-1)a_{n-i}$ for $1\leq i\leq n$ (supplemented by $A_0=1$,
$a_n=0$) and hence provide the same information about
$\mathcal{C}$. The famous \emph{MacWilliams Identities},
\cite{macwilliams63}
\begin{equation}
  \label{mac_williams_identies}
  \sum_{j=0}^{n-i} {{n-j}\choose i} A_j=q^{k-i}\cdot \sum_{j=0}^i
  {{n-j}\choose{n-i}}A_j^\perp\quad\text{for }0\le i\le n, 
\end{equation} 
relate the weight distributions $(A_i)$, $(A_i^\perp)$ of the (primal)
code $C$ and the dual code
$C^\perp=\{\vek{y}\in\F_q^n;x_1y_1+\dots+x_ny_n=0\text{ for all
}\vek{x}\in C\}$. They can be solved for $A_i$ (or $A_i^\perp$),
resulting in linear relations whose coefficients are values of 
\emph{Krawtchouk polynomials}; see,
e.g.,~\cite[Ch.~7.2]{huffman2010fundamentals} for details.
%
In our case we have $A_1^\perp=A_2^\perp=0$, since $C^\perp$ has
minimum distance $d^\perp\geq 3$, and the first three equations in
\eqref{mac_williams_identies} are equivalent to the equations in
Lemma~\ref{lemma_standard_equations_q}.

Of course the $A_i$ and the $A_i^\perp$ in (\ref{mac_williams_identies}) have to be non-negative integers. Omitting the integrality condition 
yields the so-called \emph{linear programming method}, see e.g.\ \cite[Section 2.6]{huffman2010fundamentals}, where the $A_i$ and $A_i^\perp$ 
are variables satisfying the mentioned constraints.\footnote{Typically, the $A_i^\perp$ are removed from the formulation using the 
explicit formulas based on the Krawtchouk polynomials, which may of course also be done automatically in the preprocessing step of 
a customary linear programming solver.} Given some further constraints on the weights of the code and/or the dual code, one may check 
whether the corresponding polyhedron 
contains non-negative rational solutions. In general, this is a very powerful 
approach 
and was used to compute bounds for codes with a given minimum
distance; see \cite{delsarte1972bounds,lloyd1957binary}.  Here we
consider a subset of the MacWilliams identities and use analytical
arguments.\footnote{The use of a special polynomial, like we will do,
  is well known in the context of the linear programming method, see
  e.g.~\cite[Section 18.1]{bierbrauer2005introduction}.}

By considering the average number of points per hyperplane, we can
guarantee the existence of a hyperplane containing a relatively small
number of points of $\mathcal{C}$. If this number is nonzero,
Lemma~\ref{lemma_heritable} allows us to lower-bound
$\card{\mathcal{C}}$ by induction.\footnote{This result is not \textit{new} at all. In
    \cite{beutelspacher1975partial} Beutelspacher used such an average
    argument in his upper bound on the size of partial
    spreads. Recently, N{\u{a}}stase and Sissokho used it in
    \cite[Lemma 9]{nastase2016maximum}. In coding theory it is well
    known in the context of the Griesmer bound. One may also interpret
    it as an easy implication of the first two MacWilliams identities,
    see Lemma~\ref{lemma_hyperplane_types_arithmetic_progression_2}
    and
    Corollary~\ref{cor_nonexistence_arithmetic_progression_2}.}

\begin{lemma}
  \label{lemma_average}
  Suppose that $\mathcal{C}\neq\emptyset$ in $\aspace$ is $q^r$-divisible with
  $\card{\mathcal{C}}=a\cdot q^{r+1}+b$ for some $a,b\in\mathbb{Z}$ and $y\in\mathbb{N}_{0}$ 
  with $y\equiv (q-1)b\pmod{q^{r+1}}$. Then there exists a hyperplane $H$ such
  that $\card{(\mathcal{C}\cap H)}\le (a-1)\cdot q^{r}+\frac{b+y}{q}\in\mathbb{Z}$ and
  $\card{(\mathcal{C}\cap H)}\equiv b\pmod {q^{r}}$.
\end{lemma}
\begin{proof}
  Set $n=\card{\mathcal{C}}$ and choose a hyperplane $H$ such that
  $n':= \card{(\mathcal{C}\cap H)}$ is minimal. Then, by considering the
  average number of points per hyperplane, we have
  \begin{equation*}
    n'\le 
    \frac{1}{\gauss{v}{1}{q}}\cdot \sum_{\text{hyperplane } H'}
    \card{(\mathcal{C}\cap H')}
    =n\cdot\gauss{v-1}{1}{q}/\gauss{v}{1}{q}<\frac{n}{q}
    =a\cdot q^r+\frac{b}{q}\leq a\cdot q^r+\frac{b+y}{q}.
  \end{equation*}
  Since $n'\equiv b\equiv \frac{b+y}{q}\pmod{q^r}$, this implies
  $n'\leq(a-1)q^r+\frac{b+y}{q}$.
  \qed
\end{proof}

Note that the stated upper bound does not depend on the specific choice of $a$ and $b$, i.e., 
there is no need to take a non-negative or \textit{small} $b$.   
Choosing $y$ as small as possible clearly gives the sharpest bound.\footnote{Another parameterization for $y$ 
is given by $y=qb'-b$, where $b'\in\mathbb{Z}$ with $b'\ge\frac{b}{q}$ and $b'\equiv b\pmod{q^{r+1}}$, so that $y\in\mathbb{N}_0$. 
Due to $b'=\frac{b+y}{q}$, $y$ is minimal if and only if $b'$ is minimal.}
If $b\ge 0$, which one can always achieve by suitably decreasing $a$, it is always possible to 
choose $y=(q-1)b$. However, for $q=3$, $r=2$, and $\card{\mathcal{C}}=1\cdot 3^{3}+16=43$, i.e., $a=1$ and $b=16$, 
Lemma~\ref{lemma_average} with $y=(2-1)16=16$ provides the existence of a hyperplane $H$ with 
$n'=\card{(\mathcal{C}\cap H)}\le 0\cdot 3^2+16=16$. Using $y=7$ gives
$n'\le 7$ and $n'\equiv 7\pmod {3^2}$, so that $n'=7$. Applying the argument again yields a subspace 
of co-dimension $2$ containing exactly one hole. Indeed, Equation~(\ref{eq_ste3}) is 
needed additionally in order to exclude the possibility of $n'=7$, so that $\smax_3(8,6;3)\le 248$, i.e., 
$\deficiency\ge 4$, cf.~Theorem~\ref{thm_partial_spread_4} stating the same bound. 

\begin{corollary}
  \label{corollary_iterated_average}
  Suppose that $\mathcal{C}\neq\emptyset$ in $\aspace$ is $q^r$-divisible with
  $\card{\mathcal{C}}=a\cdot q^{r+1}+b$ for some $a,b,y\in\mathbb{Z}$
  with $y\equiv (q-1)b\pmod{q^{r+1}}$ and $y\ge 1$. Further, let $g\in\mathbb{Q}$ is the largest 
  number with $q^g\mid y$, and $j\in\mathbb{Z}$ satisfies $1\leq
  j<r+1-\max\{0,g-1\}$. Then there exists a $(v-j)$-subspace $U$ such
  that $\card{(\mathcal{C}\cap U)}\le (a-j)\cdot q^{r+1-j}+\frac{b+\gauss{j}{1}{q}\cdot y}{q^j}$ and
  $\card{(\mathcal{C}\cap U)}\equiv b\pmod {q^{r+1-j}}$.
\end{corollary}
\begin{proof}
  In order to apply induction on $j$, using Lemma~\ref{lemma_heritable} and Lemma~\ref{lemma_average}, 
  we need to ensure $n'>0$ in all but the last step. The latter holds due to $pq^g\nmid b$.\footnote{The proof shows that the
  second assertion of the Corollary 
  is true for all $(v-j)$-subspaces $U$.}
  \qed
\end{proof}

Choosing the same value of $y$ in every step, in general is not the optimal way to iteratively 
apply Lemma~\ref{lemma_average}, even if $y$ is chosen optimal for the first step. To this end,  
consider a $3^3$-divisible set $\mathcal{C}\in\SPG{v-1}{\F_{3}}$ with $\card{\mathcal{C}}=31\cdot 3^4+49=2560$, 
which indeed exists as the disjoint union of $64$ solids is an example. Here $y=17$ with $y\equiv (3-1)\cdot 49\pmod{3^4}$ 
is the optimal choice in Lemma~\ref{lemma_average}, so that Corollary~\ref{corollary_iterated_average} 
guarantees the existence of a subspace $U$ with co-dimension $3$, $\card(\mathcal{C}\cap U)\le (31-3)\cdot 3^{4-3}+
\frac{49+\gauss{3}{1}{3}\cdot 17}{3^3}=94$ and $\card(\mathcal{C}\cap U)\equiv 49\equiv 1\pmod{3^1}$. However, applying 
Corollary~\ref{corollary_iterated_average} with $j=2$ and 
$y=17$ guarantees the existence of a subspace $U'$ with co-dimension $2$, $\card(\mathcal{C}\cap U')\le (31-2)\cdot 3^{4-2}+
\frac{49+\gauss{2}{1}{3}\cdot 17}{3^2}=29\cdot 3^2+13=30\cdot 3^2+4=274$, and $\card(\mathcal{C}\cap U')\equiv 49\equiv 4\pmod 9$. 
Since $\mathcal{C}\cap U'$ is $3^1$-divisible and $8\equiv 2\cdot 4\pmod{3^2}$, we can apply Lemma~\ref{lemma_average} with 
$y=8$ and deduce the existence of a hyperplane $H$ of $U'$ with $\card(\mathcal{C}\cap (U'\cap H)) \le 29\cdot 3^1 +\frac{4+8}{3}=91$ 
and $\card(\mathcal{C}\cap (U'\cap H))\equiv 4\equiv 1\pmod{3^1}$, while $U'\cap H$ has co-dimension $3$.

In the context of partial spreads or, more generally, vector space
partitions another parameterization using the number of non-hole
elements of the vector space partition turns out to be very useful in
order to state a suitable formula for $y$. In what follows we will say that a 
 vector space partition $\mathcal{P}$ of $\SV{v}{q}$ has
 \emph{hole-type} $(t,s,m_1)$ if $\mathcal{P}$ has $m_1$ holes
 ($1$-subspaces), $2\leq s\leq t<v$, and $s\leq\dim(X)\leq t$ for all
 non-holes in $\mathcal{P}$. Additionally, we assume that there is at least one non-hole.

\begin{corollary}
  \label{cor_j_times}
  Let $\mathcal{P}$ be a vector space partition of $\SV{v}{q}$ of hole-type $(t,s,m_1)$, $l,x\in\mathbb{N}_0$ with $\sum_{i=s}^t m_i =lq^s+x$, 
  and $b,c\in\mathbb{Z}$ with $m_1=bq^s+c\ge 1$. If $x\ge 2$ and $g$ is the largest integer such that $q^g$ divides $x-1$, then for each
  $0\le j\le s-\max\{1,g\}$ there exists a $(v-j)$-dimensional subspace $U$
  containing $\widehat{m}_1$ holes with
  $\widehat{m}_1\equiv \widehat{c}~\pmod {q^{s-j}}$ and
  $\widehat{m}_1\le (b-j)\cdot q^{s-j}+\widehat{c}$, where
  $\widehat{c}=\frac{c+\gauss{j}{1}{q}\cdot (x-1)}{q^j}\in\mathbb{Z}$.
\end{corollary}
\begin{proof}
  We have $\gauss{v}{1}{q} = m_1 + \sum_{i=s}^t m_i \gauss{i}{1}{q}$.
  Multiplication by $q-1$ and reduction modulo $q^{s}$ yields $-1 \equiv (q-1)c - x \pmod{q^{s}}$, allowing us to apply Corollary~\ref{corollary_iterated_average} 
  with $x = y-1$.
  Observe that the parameters $g$ from Corollary~\ref{corollary_iterated_average} 
  and Corollary~\ref{cor_j_times} differ by at most $1-1/e$ if $q=p^e$.  
  \qed
\end{proof}

So far, we can guarantee that some subspace contains not \textit{too
  many} holes, since the average number of holes per subspace would be
too large otherwise. The modulo-constraints captured in the definition
of a $q^r$-divisible set enable iterative rounding, thereby sharpening the
bounds. First we consider the special case of partial spreads, and then 
we will derive some non-existence results for vector space partitions
with \textit{few} holes.

\begin{lemma}
  \label{lemma_application_partial_spreads}
  Let $\mathcal{P}$ be a 
  vector space partition of type $k^{m_k}1^{m_1}$ of $\SV{v}{q}$ with $m_k=lq^k\!+\!x$, 
  where $l=\frac{q^{v\!-\!k}\!-\!q^r}{q^k\!-\!1}$, $x\ge 2$, $k=\gauss{r}{1}{q}\!+\!1\!-\!z\!+\!u>r$, $q^g\mid x\!-\!1$, $q^{g\!+\!1}\!\nmid\! x\!-\!1$, 
  and $g,u,z,r,x\in\mathbb{N}_0$. For $\max\{1,g\}\le y\le k$ there exists 
  a $(v-k+y)$-subspace $U$ with $L\le (z+y-1-u)q^y+w$ holes, where $w=-(x-1)\gauss{y}{1}{q}$ and $L\equiv w \pmod {q^y}$. 
\end{lemma}
\begin{proof}
  Due to 
  $m_1=\gauss{v}{1}{q}-m_k\cdot \gauss{k}{1}{q}=\gauss{r}{1}{q}q^k-\gauss{k}{1}{q}(x-1)$,  we 
  have $m_1=bq^k+c$ for $b=\gauss{r}{1}{q}$ and $c=-\gauss{k}{1}{q}(x-1)$, where $q^{g'}\mid x-1$ if and only if $q^{g'} \mid c$. Setting $s=t=k$ 
  and $j=k-y$, we observe $0\le j\le k-\max\{1,g\}$, since $\max\{1,g\}\le y\le k$.  With this, we
  apply Corollary~\ref{cor_j_times} and obtain an $(n-k+y)$-subspace $U$ with\\[-5mm]  
  \begin{eqnarray*}
    L & \le & (b-j)\cdot q^{k-j}+\frac{c+\gauss{j}{1}{q}\cdot (x-1)}{q^j}
    = (z+y-1-u)\cdot q^y-(x-1)\cdot\frac{\gauss{k}{1}{q}-\gauss{k-y}{1}{q}}{q^{k-y}}\\
    &=& (z+y-1-u)q^y-(x-1)\gauss{y}{1}{q}=(z+y-1-u)q^y+w\\[-5mm]
  \end{eqnarray*}
  holes, so that $L\le (z+y-1-u)q^y+w$ and $L\equiv w \pmod {q^y}$.
\qed\end{proof}

The parameter $l$ is chosen in such a way that $m_k=lq^k+x$ matches
the cardinality of the partial $k$-spread given by the construction in
Theorem~\ref{thm:multicomponent} for $x=1$. Thus the assumption
$x\ge 2$ is no real restriction. Actually, the chosen parameterization
using $x$ in 
Corollary~\ref{cor_j_times}
makes it very transparent why the construction of
Theorem~\ref{thm:multicomponent} is \textit{asymptotically
  optimal}---as stated in Theorem~\ref{thm_partial_spread_asymptotic}.
If the dimension $k$ of the elements of the partial spread is large
enough, a sufficient number of rounding steps can be performed while
the rounding process is stopped at $x=1$ for the other direction. For
\textit{small} $k$ we will not reach the lower bound of the
construction of Theorem~\ref{thm:multicomponent}, so that there
remains some room for better constructions.

\begin{lemma}
  \label{lemma_hyperplane_types_arithmetic_progression_c}
  Let $\Delta=q^{s-1}$, $m\in\mathbb{Z}$, and $\mathcal{P}$ be a vector space partition of $\SV{v}{q}$ of hole-type $(t,s,c)$. Then, 
  $\tau_q(c,\Delta,m)\cdot \frac{q^{v-2}}{\Delta^2}-m(m-1)\ge 0$ and $\tau_q(c,\Delta,m)\ge 0$, 
  where $\tau_q(c,\Delta,m)=m(m-1)\Delta^2q^2-c(2m-1)(q-1)\Delta q+c(q-1)\Big(c(q-1)+1\Big)$. 
  If $c>0$, then $\tau_q(c,\Delta,m)=0$ if and only if $m=1$ and $c=\gauss{s}{1}{q}$.
\end{lemma}
\begin{proof}
  Adding $(c-m\Delta)\Big(c-(m-1)\Delta\Big)$ times the first,
  $-\Big(2c-(2m-1)\Delta-1\Big)$ times the second and twice the third
  standard equation from Lemma~\ref{lemma_standard_equations_q}, and
  dividing the result by $\Delta^2/(q-1)$ gives
  $ (q-1)\cdot\sum_{h=0}^{\left\lfloor c/\Delta\right\rfloor}
  (m-h)(m-h-1)a_{c-h\Delta}=\tau_q(c,\Delta,m)\cdot
  \frac{q^{n-2}}{\Delta^2}-m(m-1), $ due to the
  $q^{s-1}$-divisibility. We observe $a_i\ge 0$ and
  $(m-h)(m-h-1)\ge 0$ for all $m,h\in\mathbb{Z}$. 
  If $m\notin\{0,1\}$, then $\tau_q(c,\Delta,m)> 0$. Solving $\tau_q(c,\Delta,0)=0$ 
  yields $c\in\left\{0,-\frac{q^s+1}{q-1}\right\}$ and solving $\tau_q(c,\Delta,1)=0$ 
  yields $c\in\left\{0,\gauss{s}{1}{q}\right\}$.
  
\qed\end{proof}

We remark that, in the case of $\tau_q(c,\Delta,m)=0$, their are either no holes at all or the holes 
form an $s$-subspace. 
\cite[Theorem 1.B]{bose1952orthogonal} is quite similar to Lemma~\ref{lemma_hyperplane_types_arithmetic_progression_c} 
and its implications. The multipliers used in the proof 
can be directly read 
off from the inverse matrix of 
  $$\mat{A}=
    \begin{pmatrix}
      1     & 1     & 1     \\
      a     & b     & c     \\
      a^2-a & b^2-b & c^2-c
    \end{pmatrix},
  $$
which is given by 
$$\mat{A}^{-1}=\frac{1}{(c-a)(c-b)(b-a)}
    \begin{pmatrix}
       bc(c-b) & -(c+b-1)(c-b) &  (c-b) \\
      -ac(c-a) &  (c+a-1)(c-a) & -(c-a) \\
       ab(b-a) & -(b+a-1)(b-a) &  (b-a)
    \end{pmatrix}
  $$
for distinct numbers $a,b,c$.
With this, 
Lemma~\ref{lemma_hyperplane_types_arithmetic_progression_c}
can be derived in a conceptual way. Consider the linear programming method with just the first three
MacWilliams identities. For parameters excluded by
Lemma~\ref{lemma_hyperplane_types_arithmetic_progression_c} this small
linear program is infeasible, which can 
be seen at a
certain basis solution, i.e., a choice of linear inequalities that are
satisfied with equality. Solving for these equations, i.e., a change
of basis, corresponds to a non-negative linear combination of the
inequality system.\footnote{If we relax $\ge 0$-inequalities by adding
  some auxiliary variable on the left hand side and the minimization
  of this variable, we can remove the infeasibility, so that we apply
  the duality theorem of linear programming. Then, the mentioned
  multipliers for the inequalities are given as the solution values of
  the dual problem.} In the parametric case we have to choose the
basis solution also depending on the parameters. Actually, we have
implemented a degree of freedom in
Lemma~\ref{lemma_hyperplane_types_arithmetic_progression_c} using the
parameter $m$. Here, the basis consists of two neighboring non-zero
$a_i$-entries, parameterized by $m$, and an arbitrary $a_i$, which
plays no role when the resulting equation is solved for all remaining
$a_i$-terms. In this way we end up with an equation of the form
$\sum_{h=0}^{\left\lfloor c/\Delta\right\rfloor}
(m-h)(m-h-1)a_{c-h\Delta}=\beta$, where the $a_i$ and their
coefficients are non-negative. 
The use of the underlying quadratic
polynomial is well known and frequently applied in the literature; see
the remarks after Theorem~\ref{thm_partial_spread_4}.

\begin{lemma}
  \label{lemma_special_excluded_vsp}
  For integers $v>k\ge s\ge 2$ and $1\le i\le s-1$, there exists no vector space partition $\mathcal{P}$ of $\SV{v}{q}$ of hole-type $(k,s,c)$, 
  where $c=i\cdot q^s-\gauss{s}{1}{q}+s-1$.\footnote{For more general non-existence results of vector space partitions see e.g. 
  \cite[Theorem 1]{heden2009length} and the related literature. Actually, we do not need the assumption of an underlying vector space 
  partition of the mentioned type. The result is generally true for $q^{s-1}$-divisible codes, since the parameter $x$ is just a nice 
  technical short-cut to ease the notation.} 
\end{lemma}
\begin{proof}
  Since we have $c<0$ for $i\le 0$, we can assume $i\ge 1$ in the following.
  Let, to the contrary, $\mathcal{P}$ be such a vector space partition and apply 
  Lemma~\ref{lemma_hyperplane_types_arithmetic_progression_c} with $m=i(q-1)$ onto $\mathcal{P}$. 
  We compute $\tau_q(c,q^{s-1},m) =\left(m-1-a\right)q^{s}+a(a+1)$ using $c(q-1)=q^s(m-1)+a$, where $a:=1+(s-1)(q-1)$.
  Setting $i=s-1-y$, we have $0\le y\le s-2$ and
  $
    \tau_q(c,q^{s-1},m)=-q^s(y(q-1)+2)+(s-1)^2q^2 -q(s-1)(2s-5) +(s-2)(s-3).
  $ 
  If $q=2$, then $y\ge 0$ and $s\ge 2$ yields\\[-2mm] 
  $$
    \tau_2(c,2^{s-1},m)=-2^s(y+2)+s^2+s  \le \left(s^2-s-2^s\right)+\left(2s-2^{s}\right) <0.
  $$    
  If $s=2$, then we have $y=0$ and $\tau_q(c,q^{s-1},m)= -q^2+q<0$. 
  If $q,s\ge 3$, then we have $q(2s-5)\ge s-3$, so that $\tau_q(c,q^{s-1},m)\le -2q^s+(s-1)^2q^2\le -2\cdot 3^{s-2}q^2+(s-1)^2q^2 $ 
  due to $y\ge 0$ and $q\ge 3$. Since $2\cdot 3^{s-2}>(s-1)^2$ for $s\ge 3$, we have $\tau_q(c,q^{s-1},m)<0$ in all cases.  
  Thus, Lemma~\ref{lemma_hyperplane_types_arithmetic_progression_c} yields a contradiction, since $q^{n-2s}>0$ and 
  $m(m-1)\ge 0$ for every integer $m$.
\qed\end{proof}

Now we are ready to present the first improved (compared to Theorem~\ref{thm_partial_spread_4}) upper bound for partial spreads, 
which also covers Theorem~\ref{thm_partial_spread_asymptotic} setting $z=0$.

\begin{theorem}
  \label{main_theorem_1}
  For integers $r\ge 1$, $t\ge 2$, $u\ge 0$, and $0\le z\le \gauss{r}{1}{q}/2$ with $k=\gauss{r}{1}{q}+1-z+u>r$ we have
  $\smax_q(v,2k;k)\le lq^k+1+z(q-1)$, where $l=\frac{q^{v-k}-q^r}{q^k-1}$ and $v=kt+r$.   
\end{theorem}
\begin{proof}
  Apply Lemma~\ref{lemma_application_partial_spreads} with $x=2+z(q-1)\ge 2$ 
  in order 
  to deduce the existence of a $(v-k+y)$-subspace $U$ with $L\le 
  (z+y-1-u)q^y-(x-1)\gauss{y}{1}{q}$ holes, where $L\equiv -(x-1)\gauss{y}{1}{q}\pmod{q^y}$. 
  Now, we set $y=z+1$.  
  Observe that $q^g\mid x-1$ implies $g\le z<y$ and we additionally have $1\le y=z+1\le \gauss{r}{1}{q}+1-z\le t$. If $z=0$, then $y=1$, $x=2$, and $L\le -uq-1 <0$. 
  For $z\ge 1$, we apply Lemma~\ref{lemma_special_excluded_vsp} 
  to the subspace $U$ with $s=y$, $c=
    (z\!+\!y\!-\!1\!-\!u)q^y-(x\!-\!1)\gauss{y}{1}{q}-jq^y = (y\!-\!1\!-\!j\!-\!u)q^y-\gauss{y}{1}{q}\!+\!y\!-\!1
  $
  for some $j\in\mathbb{N}_0$, and $i=y-1-j-u\in\mathbb{Z}$.
  Thus, $\smax_q(n,2k;k)\le lq^k+ x-1$.
\qed\end{proof}

The case $z=0$ covers Theorem~\ref{thm_partial_spread_asymptotic}. The
non-negativity of the number of holes in a certain carefully chosen 
subspace is sufficient to prove this fact.  The case $z=1$ was
announced in \cite[Lemma 10]{nastase2016maximum} and proven in
\cite{nastase2016maximumII}.  Since the known constructions for
partial $k$-spreads give $\smax_q(kt+r,2k;k)\ge lq^k+1$, see
e.g.~\cite{beutelspacher1975partial} or
Theorem~\ref{thm:multicomponent}, Theorem~\ref{main_theorem_1} is
tight for $k\ge\gauss{r}{1}{q}+1$ and $\smax_2(8,6;3)=34$.

So far Lemma~\ref{lemma_hyperplane_types_arithmetic_progression_c} was
just applied in the case of Lemma~\ref{lemma_special_excluded_vsp}
excluding the existence of some very special vector space
partitions. Next, we look at a subspace and consider the number of
holes, i.e., we apply Lemma~\ref{lemma_application_partial_spreads}
giving us the freedom to choose the dimension of the subspace.  Then
Lemma~\ref{lemma_hyperplane_types_arithmetic_progression_c}, stating
that a certain quadratic polynomial is non-negative, can be
applied. By minimizing this function in terms of the free parameter
$m$, we obtain the following result.

\newcommand{\uu}{\lambda}
\begin{theorem}
  \label{main_theorem_2}
  For integers $r\ge 1$, $t\ge 2$, $y\ge \max\{r,2\}$, 
  $z\ge 0$ with $\uu=q^{y}$, $y\le k$, 
  $k=\gauss{r}{1}{q}+1-z>r$, $v=kt+r$, and  $l=\frac{q^{v-k}-q^r}{q^k-1}$, we have $$\smax_q(v,2k;k)\le 
       lq^k+\left\lceil \uu -\frac{1}{2}-\frac{1}{2}
    \sqrt{1+4\uu\left(\uu-(z+y-1)(q-1)-1\right)} \right\rceil.
  $$   
\end{theorem}
\begin{proof}
  From Lemma~\ref{lemma_application_partial_spreads} we conclude $L\le (z+y-1)q^y-(x-1)\gauss{y}{1}{q}$ and $L\equiv -(x-1)\gauss{y}{1}{q} 
  \pmod {q^y}$ for the number of holes of a certain $(v\!-\!k\!+\!y)$-subspace $U$. 
  Using the notation 
  of Lemma~\ref{lemma_application_partial_spreads}, $\mathcal{P}\cap U:=\{P\cap U\mid P\in\mathcal{P}\}$ 
  is of hole-type $(k,y,L)$ if $y\ge 2$. 
  Next, we will show that 
  $\tau_q(c,\Delta,m)\le 0$, where $\Delta=q^{y-1}$ and $c=iq^{y}-(x-1)\gauss{y}{1}{q}$ with $1\le i\le z+y-1$, for suitable integers $x$ and $m$. 
  Note that, in order to apply Lemma~\ref{lemma_application_partial_spreads}, we have to satisfy $x\ge 2$ and $y\ge g$ for all integers $g$ 
  with $q^g\mid x-1$. Applying Lemma~\ref{lemma_hyperplane_types_arithmetic_progression_c} then gives the desired contradiction, so that 
  $\smax_q(n,2k;k)\le lq^k+x-1$.
  
  We choose\footnote{
  Solving $\frac{\partial \tau_q(c,\Delta,m)}{\partial m}=0$, i.e., minimizing $\tau_q(c,\Delta,m)$, yields 
  $m=i(q-1)-(x-1)+\frac{1}{2}+\frac{x-1}{q^y}$. 
  For $y\ge r$ we can assume $x-1<q^y$ due to Theorem~\ref{thm:multicomponent}, so that up-rounding yields the optimum integer choice. For $y<r$ 
  the interval $\left[\uu+\frac{1}{2}- \frac{1}{2}\theta(i),\uu+\frac{1}{2}+ \frac{1}{2}\theta(i)\right]$ may contain no integer. 
  } $m=i(q-1)-(x-1)+1$, so that $\tau_q(c,\Delta,m)=x^2-(2\uu+1)x+
  \uu(i(q-1)+2)$. Solving $\tau_q(c,\Delta,m)=0$ for $x$ gives 
  $x_0=\uu+\frac{1}{2}\pm \frac{1}{2}\theta(i)$,  
  where $$\theta(i)=\sqrt{1-4i\uu(q-1)+4\uu(\uu-1)}.$$ We have $\tau_q(c,\Delta,m)\le 0$ for 
  $\left|2x-2\uu-1\right|\le \theta(i)$. We need to find an integer $x\ge 2$ such that this inequality is satisfied for all
  $1\le i\le z+y-1$. The strongest restriction is attained for $i=z+y-1$. Since $z+y-1\le\gauss{r}{1}{q}$ and $\uu=q^y\ge q^r$, we have 
  $\theta(i)\ge\theta(z+y-1)\ge 1$, so that $\tau_q(c,\Delta,m)\le 0$ for 
  $x= \left\lceil \uu+\frac{1}{2}- \frac{1}{2}\theta(z+y-1)\right\rceil$. 
  (Observe $x\le \uu+\frac{1}{2}+ \frac{1}{2}\theta(z+y-1)$ due 
  to $\theta(z+y-1)\ge 1$.) 
  Since $x\le \uu+1$, we have $x-1\le \uu=q^y$, so that $q^g \mid x-1$ implies $g\le y$ provided $x\ge 2$. The latter is true due to 
  $\theta(z+y-1)\le \sqrt{1-4\uu(q-1)+4\uu(\uu-1)}\le \sqrt{1+4\uu(\uu-2)}< 2(\uu-1)$, which implies $x\ge\left\lceil\frac{3}{2}\right\rceil=2$.
  
  So far we have constructed a suitable $m\in\mathbb{Z}$ such that $\tau_q(c,\Delta,m)\le 0$ for 
  $$x=\left\lceil \uu+\frac{1}{2}- \frac{1}{2}\theta(z+y-1)\right\rceil.$$ If $\tau_q(c,\Delta,m)< 0$, then 
  Lemma~\ref{lemma_hyperplane_types_arithmetic_progression_c} gives a contradiction, so that we assume 
  $\tau_q(c,\Delta,m)=0$ in the following. If $i<z+y-1$ we have $\tau_q(c,\Delta,m)<0$ due to $\theta(i)>\theta(z+y-1)$, so that we 
  assume $i=z+y-1$. Thus, $\theta(z+y-1)\in\mathbb{N}_0$. However, we can write  
  $
    \theta(z+y-1)^2=1+4\uu\left(\uu-(z+y-1)(q-1)-1\right)=(2w-1)^2 =1+4w(w-1)
  $
  for some integer $w$. If $w\notin\{0,1\}$, then $\gcd(w,w-1)=1$, so that either $\uu=q^y\mid w$ or $\uu=q^y\mid w-1$. Thus, in any case, $w\ge q^y$, which is 
  impossible since $(z+y-1)(q-1)\ge 1$. Finally, $w\in\{0,1\}$ implies $w(w-1)=0$, so that $\uu-(z+y-1)(q-1)-1=0$. Thus, 
  $z+y-1=\gauss{y}{1}{q}\ge \gauss{r}{1}{q}$ since $y\ge r$. The assumptions $y\le k$ and $k=\gauss{r}{1}{q}+1-z$ imply
  $z+y-1=\gauss{r}{1}{q}$ and $y=r$. This gives $k=r$, which is excluded.
\qed
\end{proof}

An example where Theorem~\ref{main_theorem_2} is strictly superior to the results of \cite[Theorem 6,7]{nastase2016maximumII} is given by 
$\smax_3(15,12;6)\le 19695$.\footnote{
For $2\le q\le 9, 1\le v,k\le 100$ the bounds of 
\cite[Theorem 6,7]{nastase2016maximumII} are covered by Theorem~\ref{main_theorem_2} and Corollary~\ref{cor_main_theorem_1}. In many cases the bounds coincide.}
Setting $y=k$, we obtain Theorem~\ref{thm_partial_spread_4}. Compared to \cite{bose1952orthogonal,nets_and_spreads}, the new ingredients 
essentially are $q^r$-divisible sets and Corollary~\ref{cor_j_times}, 
which allows us to choose $y<k$. 
Theorem~\ref{thm_partial_spread_4}, e.g., gives $\smax_2(15,12;6)\le 516$, $\smax_2(17,14;7)\le 1028$, and $\smax_9(18,16;8)\le 3486784442$, 
while Theorem~\ref{main_theorem_2} gives $\smax_2(15,12;6)\le 515$, $\smax_2(17,14;7)\le 1026$, and $\smax_9(18,16;8)\le 3486784420$.
For $2\le q\le 9$, $5\le k\le 19$, there are $66$ improvements in total, i.e., almost $19\%$, and the maximum gap is $22$. Next, we provide an 
estimation of the bound of Theorem~\ref{thm_partial_spread_4}.

\begin{lemma}
  For integers $1\le r<k$ and $q\ge 2$ we have 
  $$
    2q^k-q^r-\frac{q^{2r-k}}{\underline{b}}
    <\sqrt{1+4q^k(q^k-q^r)}\le 
    2q^k-q^r-\frac{q^{2r-k}}{\overline{b}},
  $$
  where $\underline{b}=\frac{3+2\sqrt{2}}{2}>2.91$ and $\overline{b}=\frac{16}{3}<5.34$.
\end{lemma}
\begin{proof}
Due to $q\ge 2$ and $k\ge r+1\ge 2$, we have\\[-9mm] 
\begin{eqnarray*}
  &&1+4q^k(q^k-q^r) > 4q^k(q^k-q^r) -q^{2r}\cdot\overset{\ge 0}{\overbrace{\left(\frac{4}{\underline{b}}-1-\frac{2}{\underline{b}q}-\frac{1}{\underline{b}^2 q^2}\right)}}  \\
  &&\ge 4q^k(q^k-q^r)-\frac{4}{\underline{b}}q^{2r}+q^{2r}+\frac{2}{\underline{b}}q^{3r-k}+\frac{1}{\underline{b}^2}q^{4r-2k}
  = \left(2q^k-q^r-\frac{q^{2r-k}}{\underline{b}}\right)^2.
\end{eqnarray*}
Similarly, 
$
  1\!+\!4q^k(q^k\!-\!q^r) \le  4q^k(q^k\!-\!q^r) -q^{2r}\cdot 
  \left(\frac{4}{\overline{b}}-1\right)
  \le\left(2q^k\!-\!q^r\!-\!\frac{q^{2r\!-\!k}}{\overline{b}}\right)^2
$.
\qed\end{proof}

\begin{corollary}
  \label{cor_df_estimation}
  For integers $1\le r<k$ and $t\ge 2$ we have $\smax_q(kt+r,2k;k)< lq^k+\frac{q^r}{2}+\frac{1}{2}+\frac{q^{2r-k}}{3+2\sqrt{2}}$, where $l=\frac{q^{(t-1)k+r}-q^r}{q^k-1}$. 
  If $k\ge 2r$, then $\smax_q(kt+r,2k;k)< lq^k+1+\frac{q^r}{2}$.
\end{corollary}  
  
\begin{corollary}
  \label{cor_main_theorem_1}
  For integers $r\ge 1$, $t\ge 2$, and $u,z\ge 0$ with $k=\gauss{r}{1}{q}+1-z+u>r$ we have
  $\smax_q(v,2k;k)\le lq^k+1+z(q-1)$, where $l=\frac{q^{v-k}-q^r}{q^k-1}$ and $v=kt+r$.   
\end{corollary}
\begin{proof}
  Using Corollary~\ref{cor_df_estimation}, we can remove the upper bound $z\le \gauss{r}{1}{q}/2$ from Theorem~\ref{main_theorem_1}. If $z>\gauss{r}{1}{q}/2$, then 
  $z\ge \gauss{r}{1}{q}/2 +1/2$, so that 
  $
    \smax_q(v,2k;k)< lq^k+1+\frac{q^r}{2}\le lq^k+1+\frac{q^r-1}{2}+\frac{q-1}{2}\le lq^k+1+z(q-1)
  $ 
  for $k\ge 2r$. Thus, we can assume $r+1\le k\le 2r-1$ and $r\ge 2$. With this, we have $z\ge\gauss{r}{1}{q}-2(r-1)$ and  $lq^k+1+z(q-1)\ge lq^k+q^r-2(q-1)(r-1)$. 
  It remains to show $lq^k+q^r-2(q-1)(r-1)\ge lq^k+\frac{q^r}{2}+\frac{1}{2}+\frac{q^{2r-k}}{3+2\sqrt{2}} \ge lq^k+\frac{q^r}{2}+\frac{1}{2}+\frac{q^{r-1}}{3+2\sqrt{2}}$, i.e., 
  $q^r\ge 1+\frac{2q^{r-1}}{3+2\sqrt{2}}+4(q-1)(r-1)$. The latter inequality is valid for all pairs $(r,q)$ except $(2,2)$, $(2,3)$, and $(3,2)$. In those cases 
  it can be verified directly that $lq^k+1+z(q-1)$ is not strictly less than the upper bound of Theorem~\ref{thm_partial_spread_4}. Indeed, both bounds coincide. 
\qed\end{proof}
  
We remark that the first part of Corollary~\ref{cor_df_estimation} can be written as $\deficiency\ge \frac{q^r-1}{2}-\frac{q^{2r-k}}{3+2\sqrt{2}}$. Unfortunately, 
Theorem~\ref{main_theorem_2} is not capable to obtain $\deficiency\ge \left\lfloor(q^r-1)/2\right\rfloor$. For $\smax_2(17,12;6)$, i.e., $q=2$ and $r=5$, 
it gives $\deficiency\ge 13$ while $\left\lfloor(q^r-1)/2\right\rfloor=15$. In Lemma~\ref{lemma_implication_fourth_mac_williams} we give a cubic 
analog to Lemma~\ref{lemma_hyperplane_types_arithmetic_progression_c}, which yields $\deficiency\ge 14$ for these parameters.
  

\section{Constructions for $\bm{q^r}$-divisible sets}
\label{sec_constructions}

First note that we can embed every $\Delta$-divisible point set
$\mathcal{C}$ in $\aspace$ into ambient spaces with dimension larger
than $v$ and, conversely, replace $\F_q^v$ by the span
$\langle\mathcal{C}\rangle$ without destroying the
$\Delta$-divisibility. Since in this sense $v$ is not determined by
$\mathcal{C}$, we will refer to $\mathcal{C}$ as a
\emph{$\Delta$-divisible point set over $\F_q$}. 
In the sequel we
develop a few basic constructions of $q^r$-divisible sets. For the
statement of the first lemma recall our convention that subspaces of
$\F_q^v$ are identified with subsets of the point set $\points$ of
$\aspace$.
\begin{lemma}
  \label{lemma_subspace}
  Every $k$-subspace $\mathcal{C}$ of $\aspace$ with $k\geq 2$ is
  $q^{k-1}$-divisible.
\end{lemma}   
\begin{proof}
  By the preceding remark we may assume $k=v$ and hence
  $\mathcal{C}=\points$. In this case the result is clear,
  since $\card{\points}-\card{H}=q^{v-1}$ for each hyperplane $H$.
\qed\end{proof}
In fact a $k$-subspace of $\F_q^v$ is associated to the
$k$-dimensional simplex code over $\F_q$ and
Lemma~\ref{lemma_subspace} is well-known.

For a point set $\mathcal{C}$ in $\aspace$ we denote by
$\chi_{\mathcal{C}}$ its characteristic function, i.e., 
$\chi_{\mathcal{C}}\colon\points\to\{0,1\}\subset\mathbb{Z}$ 
with $\chi_{\mathcal{C}}(P)=1$ if and only if $P\in\mathcal{C}$.

\begin{lemma}
  \label{characteristic sum}
  Let $\mathcal{C}_i$ be $\Delta_i$-divisible point sets in $\aspace$
  and $a_i\in\mathbb{Z}$ for
  $1\le i\le m$. If $\mathcal{C}\subseteq\points$ satisfies
  $\chi_{\mathcal{C}}=\sum_{i=1}^m a_i\chi_{\mathcal{C}_i}$ then
  $\mathcal{C}$ is $\gcd(a_1\Delta_1,\dots,a_m\Delta_m)$-divisible.  
\end{lemma}
\begin{proof}
  We have
  $\card{\mathcal{C}}=\sum_{i=1}^ma_i\cdot\card{\mathcal{C}_i}$ and
  $\card{(\mathcal{C}\cap H)}=\sum_{i=1}^ma_i\cdot
  \card{(\mathcal{C}_i\cap H)}$ for each hyperplane $H$.  Since
  $\card{(\mathcal{C}_i\cap H)}\equiv
  \card{\mathcal{C}_i}\pmod{\Delta_i}$, the result follows.
\qed\end{proof}
Lemma~\ref{characteristic sum} shows in particular that the union
of mutually disjoint $q^r$-divisible sets is again
$q^r$-divisible. Another (well-known) corollary is the following,
which expresses the divisibility properties of the MacDonald
codes.\footnote{The generalization to more than one ``removed'' subspace is
  also quite obvious and expresses the divisibility properties of
  optimal linear codes of type BV in the projective case
  \cite{belov-etal74,hill92,ivan00}.}
\begin{corollary}
  \label{cor_affin}
  Let $X\subsetneq Y$ be subspaces of $\F_q^v$ and  
  $\mathcal{C}=Y\setminus X$. If $\dim(X)=s$ then $\mathcal{C}$ is
  $q^{s-1}$-divisible.
\end{corollary}
In particular affine $k$-subspaces of $\F_q^v$ are $q^{k-1}$ divisible.
\begin{lemma}
  \label{lemma_sum}
  Let $\mathcal{C}_1\in\points_1$, $\mathcal{C}_2\in\points_2$
  be $q^r$-divisible point sets in $\SPG{v_1-1}{\mathbb{F}_q}$, respectively,
$\SPG{v_2-1}{\mathbb{F}_q}$. Then
  there exists a $q^r$-divisible set $\mathcal{C}$ in
  $\SPG{v_1+v_2-1}{\mathbb{F}_q}$ with $\card{\mathcal{C}}=
  \card{\mathcal{C}_1}+\card{\mathcal{C}_2}$.
\end{lemma}
\begin{proof}
  Embed the point sets $\mathcal{C}_1,\mathcal{C}_2$ in the obvious way into
  $\PG(\F_q^{v_1}\times\F_q^{v_2})\cong
  \SPG{v_1+v_2-1}{\mathbb{F}_q}$, and
  take $\mathcal{C}$ as their union.
\qed\end{proof}
Let us note that the embedding dimension $v$ in Lemma~\ref{lemma_sum}
is usually not the smallest possible, and the isomorphism type of
$\mathcal{C}$ is usually not determined by $\mathcal{C}_1$ and
$\mathcal{C}_2$.\footnote{If not both $\mathcal{C}_1$ and
  $\mathcal{C}_2$ are subspaces, 
then disjoint embeddings into a geometry $\SPG{v-1}{\mathbb{F}_q}$
with $v<\dim\langle\mathcal{C}_1\rangle+\dim\langle\mathcal{C}_2\rangle$
exist as well.} 

In analogy to the \emph{Frobenius Coin Problem},
cf.~\cite{beutelspacher1978partitions,brauer1942problem,heden1984frobenius},
we define $\frobenius(q,r)$ as the smallest positive integer such that
a $q^r$-divisible set over $\mathbb{F}_q$ (i.e., with some ambient
space $\F_q^v$) with cardinality $n$ exists for all integers
$n>\frobenius(q,r)$. Using Lemma~\ref{lemma_subspace},
Corollary~\ref{cor_affin}, and Lemma~\ref{lemma_sum}, we conclude that
$\frobenius(q,r)\le \gauss{r+1}{1}{q}\cdot q^{r+1}-\gauss{r+1}{1}{q}-
q^{r+1}$, the largest integer
not representable as $a_1\gauss{r+1}{1}{q}+a_2q^{r+1}$ with
$a_1,a_2\in\mathbb{Z}_{\geq 0}$.\footnote{Note that
  $\gcd\!\left(\gauss{r+1}{1}{q},q^{r+1}\right)=1$ and recall the
  solution of the ordinary Frobenius Coin Problem.} The bound may also
be stated as $\frobenius(q,r)\leq\sum_{i=r+2}^{2r+1}q^i-\sum_{i=0}^rq^i$.

As the disjoint union of $q^r$-divisible sets is again $q^r$-divisible, one obtains 
a wealth of constructions. Consequently, the $q^r$-divisible point sets not arising in this way are of particular interest. They are called \emph{indecomposable}.

The next construction uses the concept of a
``sunflower'' of subspaces, which forms the $q$-analogue of
the $\Delta$-systems, or sunflowers, considered in extremal set theory
\cite{erdos-rado60}.\footnote{Our sunflowers need not have constant
dimension, however.}
\begin{definition}
  Let $X$ be a subspace of $\F_q^v$ and $t\geq 2$ an integer.  A
  \emph{$t$-sunflower} in $\F_q^v$ with \emph{center} $X$ is a set
  $\{Y_1,\dots,Y_t\}$ of subspaces of $\F_q^v$ satisfying $Y_i\neq X$
  and $Y_i\cap Y_j=X$ for $i\neq j$. The point sets $Y_i\setminus X_i$
  are called \emph{petals} of the sunflower.
\end{definition}
\begin{lemma}
  \label{lma:sunflower}
  \begin{enumerate}[(i)]
  \item The union of the petals of a $q$-sunflower in $\F_q^v$ with
    $r$-dimensional center forms a $q^r$-divisible point set.
  \item The union of the petals and the center
    of a $q+1$-sunflower in $\F_q^v$ with $r$-dimensional center forms
    a $q^r$-divisible point set.
  \end{enumerate}
\end{lemma}
\begin{proof}
  (i) Let $\mathcal{F}=\{Y_1,\dots,Y_q\}$ and
  $\mathcal{C}=\bigcup_{i=1}^q(Y_i\setminus
  X_i)=\left(\bigcup\mathcal{F}\right)\setminus X$. We have
  $\chi_{\mathcal{C}}=\sum_{i=1}^q\chi_{Y_i}-q\chi_X$. Since
  $\dim(Y_i)\geq r+1$, $Y_i$ is $q^r$-divisible, and so is
  $q\chi_X$. Hence, by Lemma~\ref{characteristic sum}, $\mathcal{C}$ is
  $q^r$-divisible as well.
  
  (ii) follows from (i) by adding one further space $Y_{q+1}$ to
  $\mathcal{F}$. 
\qed\end{proof}
\begin{lemma}
  \label{lma:construction1}
  Let $r\ge 1$ be an integer and $1\le i\le q^r+1$. There exists a
  $q^r$-divisible set $\mathcal{C}_i$ over $\F_q$ with
  $\card{\mathcal{C}_i}
  =\gauss{2r}{1}{q}+i\cdot\left(q^{r+1}-q^r-\gauss{r}{1}{q}\right)$.
\end{lemma}
\begin{proof}
  Let $Y=\F_q^{2r}$ and $X_1,\dots,X_{q^r+1}$ an
  $r$-spread in $Y$. After embedding $Y$ in a space $\F_q^v$ of
  sufficiently large dimension $v$, it is possible to choose
  $q$-sunflowers $\mathcal{F}_1,\dots,\mathcal{F}_{q^r+1}$ in $\F_q^v$
  with the following properties: $Y\in\mathcal{F}_i$ for all $i$;
  $\dim(Z)=r+1$ for $Z\in\mathcal{F}_i\setminus\{Y\}$;
  $\mathcal{F}_i$ has center $X_i$; petals in different sunflowers
  $\mathcal{F}_i$ and $\mathcal{F}_j$ are either equal (to $Y$)
  or disjoint. Having made such a choice, we set
  $\mathcal{C}_i
  =\left(\bigcup\mathcal{F}_1\cup\dots\cup\bigcup\mathcal{F}_i\right)
  \setminus(X_1\cup\dots\cup X_i)$ for $1\leq i\leq q^r+1$. Then
  $\chi_{\mathcal{C}_i}=\sum_{Z\in\mathcal{F}_1\cup\dots\cup\mathcal{F}_i}\chi_Z
  -q\chi_{X_1}-\dots-q\chi_{X_i}$ is $q^r$-divisible (again by
  Lemma~\ref{characteristic sum}), and $\card{\mathcal{C}_i}$ is as
  asserted.
\qed\end{proof}

Replacing $S$ by some arbitrary $q^1$-divisible set, we similarly obtain:
\begin{lemma}
  Let $\mathcal{C}'$ be a $q^1$-divisible set of cardinality $n$, then
  there exist $q^1$-divisible sets of
  cardinality $n+i\cdot\left(q^2-q-1\right)$ for all $0\le i\le n$.  
\end{lemma}
%

Our last construction in this section uses the concept of a cone.

\begin{definition}
  \label{dfn:cone}
  Let $X$, $Y$ be complementary subspaces of $\F_q^v$
  with $\dim(X)=s$, $\dim(Y)=t$ (hence $v=s+t$)
  and $\mathcal{B}$ a set of points in $\PG(Y)$. The \emph{cone} with
  \emph{vertex} $X$ and 
  \emph{base} $\mathcal{B}$ is the point set
  $\mathcal{C}=\bigcup_{P\in\mathcal{B}}(X+P)$.
\end{definition}
\begin{lemma}
  \label{lma:cone}
  Let $\mathcal{B}$ be a $q^r$-divisible point set in $\PG(v-1,\F_q)$
  with $\card{\mathcal{B}}=m$ and $s\geq 1$ an integer.
  \begin{enumerate}[(i)]
  \item If $m\equiv 0\pmod{q^{r+1}}$ then there exists a
    $q^{r+s}$-divisible point set $\mathcal{C}$ in $\PG(v+s-1,\F_q)$
    of cardinality $\card{\mathcal{C}}=mq^s$.
  \item If $m(q-1)\equiv -1\pmod{q^{r+1}}$ then there exists a
    $q^{r+s}$-divisible point set $\mathcal{C}$ in $\PG(v+s-1,\F_q)$
    of cardinality $\card{\mathcal{C}}=\gauss{s}{1}{q}+mq^s$.
  \end{enumerate}
\end{lemma}
\begin{proof}
  Embed $\F_q^v$ into $\F_q^{v+s}$ as $Y$ and consider a cone
  $\mathcal{K}$ in $\PG(v+s-1,\F_q)$ with base $\mathcal{B}$ and
  $s$-dimensional vertex $X$. The hyperplanes $H\supseteq X$ satisfy
  $\card{(\mathcal{K}\setminus H)}=q^s\cdot\card{(\mathcal{K}\setminus
    H)}\equiv0\pmod{q^{r+s}}$. The hyperplanes $H\nsupseteq X$
  intersect $X+P$, $P\in\mathcal{B}$, in an $s$-subspace $\neq X$,
  hence contain $\gauss{s}{s-1}{q}=$ points in $X$ and $q^{s-1}$ points
  in $\mathcal{K}\setminus X$. It follows that $\card{(\mathcal{K}\setminus
    H)}=q^{s-1}+m(q^s-q^{s-1})=(1+m(q-1))q^{s-1}$. Thus in Case~(ii) we
  can take $\mathcal{C}=\mathcal{K}$ and in Case~(i) we can take
  $\mathcal{C}=\mathcal{K}\setminus X$.\footnote{Note that
    $m(q-1)\equiv 0\pmod{q^{r+1}}$ is equivalent to $m\equiv
    0\pmod{q^{r+1}}$.}
\qed\end{proof}
The preceding constructions can be combined in certain nontrivial ways
to yield further constructions of $q^r$-divisible point sets.  We will
return to this topic in Section~\ref{ssec:better}. 

Nonetheless, we are
only scratching the surface of a vast subject. Projective two-weight
codes with weights $w_1,w_2$ satisfying $w_2>w_1+1$ are
$q^r$-divisible by Delsarte's Theorem \cite[Cor.~2]{delsarte72}. This
yields many further examples of $q^r$-divisible point sets; see
\cite{calderbank1986geometry} and the online-table at
\url{http://moodle.tec.hkr.se/\textasciitilde
  chen/research/2-weight-codes}. Codes meeting the Griesmer bound
whose minimum distance is a multiple of $q$ are $q^r$-divisible
\cite[Prop.~13]{ward2001divisible_survey}.\footnote{In the case $q=p$,
  and in general for codes of type BV, such codes are even
  $q^e$-divisible, where $q^e$ is the largest power of $p$ dividing
  the minimum distance \cite[Th.~1 and
  Prop.~2]{ward1998divisibility}.}  Optimal codes of lengths strictly
above the Griesmer bound tend to have similar divisibility properties;
see, e.g., \emph{Best Known Linear Codes} in \texttt{Magma}.

\section{More non-existence results for $\bm{q^r}$-divisible 
sets}\label{sec_non_existence}

  For a point set $\mathcal{C}$ in $\aspace$ let
  $\mathcal{T}(\mathcal{C}):=\left\{0\le i\le c\mid a_i>0\right\}$,
  where $a_i$ denotes the number of hyperplanes with
  $\card{(\mathcal{C}\cap H)}=i$.

\begin{lemma}
  \label{lemma_hyperplane_types_arithmetic_progression_2}
  For integers $u\in\mathbb{Z}$, $m\ge 0$ and $\Delta\ge 1$ let
  $\mathcal{C}$ in $\aspace$ be $\Delta$-divisible of cardinality
  $n=u+m\Delta\ge 0$.  Then, we have
  $(q-1)\cdot\sum_{h\in\mathbb{Z},h\le m}
  ha_{u+h\Delta}=\left(u+m\Delta-uq\right)\cdot
  \frac{q^{v-1}}{\Delta}-m $, where we set $a_{u+h\Delta}=0$ if
  $u+h\Delta<0$.
\end{lemma}
\begin{proof}
  Rewriting the equations from Lemma~\ref{lemma_standard_equations_q}  
  yields 
    $(q-1)\cdot\sum_{h\in\mathbb{Z},h\le m} a_{u+h\Delta} = q\cdot q^{v-1}-1 $ and $
    (q-1)\cdot\sum_{h\in\mathbb{Z},h\le m} (u+h\Delta)a_{u+h\Delta} = (u+m\Delta)(q^{v-1}-1)$.
  $u$ times the first equation minus the second equation gives $\Delta$ times the stated equation. 
\qed\end{proof}
 
\begin{corollary}
  \label{cor_nonexistence_arithmetic_progression_2}
  Let $\mathcal{C}$ in $\aspace$ satisfy 
  $n=\card{\mathcal{C}}=u+m\Delta$ 
  and $\mathcal{T}(\mathcal{C})\subseteq\{u,u+\Delta,\dots,u+m\Delta\}$. 
  Then $u<\frac{n}{q}$ or $u=n=0$.
\end{corollary}

While the quadratic inequality of
Lemma~\ref{lemma_hyperplane_types_arithmetic_progression_c} is based
on the first three MacWilliams identities, the linear inequality of
Lemma~\ref{lemma_hyperplane_types_arithmetic_progression_2} is based
on the first two MacWilliams identities. Corollary~\ref{cor_nonexistence_arithmetic_progression_2}
corresponds to the average argument that we have used in the proof of
Lemma~\ref{lemma_average}. Lemma~\ref{lemma_hyperplane_types_arithmetic_progression_c} can of course be applied in the general case of $q^r$-divisible sets. 
First we characterize the part of the parameter space where $\tau_q(c,\Delta,m)\le 0$ and then we analyze the right side of
the corresponding interval, 

\begin{lemma}
  \label{lemma_negative_tau}    
  For 
  $m\in\mathbb{Z}$, we have $\tau_q(c,\Delta,m)\le 0$ if and only if $(q-1)c-(m-1/2)\Delta q+\frac{1}{2}\in$ 
  \begin{equation}
    \label{ie_forbidden_interval}
    \left[-\frac{1}{2}\cdot\sqrt{q^2\Delta^2-4qm\Delta+2q\Delta+1},
    \frac{1}{2}\cdot\sqrt{q^2\Delta^2-4qm\Delta+2q\Delta+1}\right].
  \end{equation}
  The last interval is non-empty, i.e., the radicand is non-negative
  if and only if 
  $m\le \left\lfloor(q\Delta+2)/4\right\rfloor$.  
  We have $\tau_q(u,\Delta,1)=0$ if and only if $u=(\Delta q-1)/(q-1)$ or $u=0$.
\end{lemma}
\begin{proof}
  Solving $\tau_q(c,\Delta,m)=0$ for $c$ yields the boundaries for $c$
  stated in \eqref{ie_forbidden_interval}).  Inside this interval we
  have $\tau_q(c,\Delta,m)\le 0$. Now,
  $q^2\Delta^2-4qm\Delta+2q\Delta+1\ge 0$ is equivalent to
  $m\le \frac{q\Delta}{4}+\frac{1}{2}+\frac{1}{4q\Delta}$. Rounding
  downward the right-hand side, while observing
  $\frac{1}{4q\Delta}< \frac{1}{4}$, yields
  $\left\lfloor(q\Delta+2)/4\right\rfloor$.  \qed\end{proof}

\begin{lemma}
  \label{lemma_m_bound_rhs}
  For $1\le m\le \left\lfloor\sqrt{(q-1)q\Delta}-q+\frac{3}{2}\right\rfloor$, we have 
  $
    (q-1)n-(m-1/2)\Delta q+\frac{1}{2}\le\frac{1}{2}\cdot\sqrt{q^2\Delta^2-4qm\Delta+2q\Delta+1}
  $,
  where $n=m\cdot \gauss{r+1}{1}{q}-1$ and $\Delta=q^r$.
\end{lemma}
\begin{proof}
  Plugging in 
  yields $\frac{1}{2}\cdot\left(q\Delta+3-2m-2q\right)\le\frac{1}{2}\sqrt{q^2\Delta^2-(4m-2)q\Delta+1}$, so
  that squaring and simplifying gives
  $m\le \sqrt{(q-1)q\Delta+1/4}-q+\frac{3}{2}$.  \qed\end{proof}

\begin{theorem}
  \label{thm_exclusion_r_1_to_ovoid}
  Let $\mathcal{C}$ in $\aspace$ be $q^1$-divisible with
  $2\le n=\card{\mathcal{C}}\le q^2$, then either $n=q^2$ or $q+1$
  divides $n$.
\end{theorem}
\begin{proof}
  First we show $n\notin \left[(m-1)(q+1)+2,m(q+1)-1\right]$ for $1\le m\le q-1$. For $m=1$ this 
  statement follows from Lemma~\ref{lemma_negative_tau} and Lemma~\ref{lemma_hyperplane_types_arithmetic_progression_c}. 
  For $m\ge 2$ let $(m-1)(q+1)+2\le n\le m(q+1)-1$. Due to Lemma~\ref{lemma_hyperplane_types_arithmetic_progression_c} 
  it suffices to verify $\tau_{q}(n,q,m)\le 0$. From $n\ge (m-1)(q+1)+2$ we conclude
  \begin{eqnarray*}
  &&(q-1)n-(m-1/2)\Delta q+\frac{1}{2} \ge -\frac{1}{2}\cdot\left(q^2-4q+1+2m\right)\ge
   -\frac{1}{2}\cdot
   \left(q^2-2m-3\right)\\
 &&\ge -\frac{1}{2}\cdot\sqrt{q^4-4mq^2+2q^2+1} 
  = -\frac{1}{2}\cdot\sqrt{q^2\Delta^2-4qm\Delta+2q\Delta+1}
  \end{eqnarray*}
  and from $n\le  m(q+1)-1$ we conclude
  $$
    (q\!-\!1)n-(m\!-\!1/2)\Delta q+\frac{1}{2}\le \frac{1}{2}\cdot\left(q^2\!-\!2m\! -\!2q\!+\!3\right)
    \overset{\star}{\le} 
    \frac{1}{2}\cdot\sqrt{q^2\Delta^2\!-\!4qm\Delta\!+\!2q\Delta\!+\!1}.
  $$
  With respect to the estimation $\star$, we remark that 
  $
    -4q^3+8q^2-12q+8 +4m(m+2q-3) \overset{m\le q-1}{\le} -4(q-1)(q^2-4q+6)\overset{q\ge 2}{\le} 0 
  $.
  Thus, Lemma~\ref{lemma_negative_tau} gives $\tau_{q}(n,q,m)\le 0$.
  
  Applying Corollary~\ref{cor_nonexistence_arithmetic_progression_2} with $u=m$ and $\Delta=q$ 
  yields $n\neq (m-1)(q+1)+1$ for all $1\le m\le q-1$.
\qed\end{proof}

The existence of ovoids 
shows that the upper bound $n\le q^2$ is sharp in Theorem~\ref{thm_exclusion_r_1_to_ovoid}.

\begin{theorem}
  \label{thm_exclusion_q_r}
  For the cardinality $n$ of a $q^r$-divisible set $\mathcal{C}$ over
  $\F_q$ we have
  $$
    n\notin\left[(a(q-1)+b)\gauss{r+1}{1}{q}+a+1,(a(q-1)+b+1)\gauss{r+1}{1}{q}-1\right],
  $$
  where $a,b\in\mathbb{N}_0$ with $b\le q-2$, $a\le r-1$, and $r\in\mathbb{N}_{>0}$.
  
  In other words, if $n\le rq^{r+1}$, 
  then $n$ can be written as
$a\gauss{r+1}{1}{q}+bq^{r+1}$ for some $a,b\in\mathbb{N}_0$.
\end{theorem}
\begin{proof}
  We prove by induction on $r$, set $\Delta=q^r$, and write $n=(m-1)\gauss{r+1}{1}{q}+x$, where $a+1\le x\le \gauss{r+1}{1}{q}-1$ and 
  $m-1=a(q-1)+b$ for integers $0\le b\le q-2$, $0\le a\le r-1$. 
  The induction start $r=1$ is given by Theorem~\ref{thm_exclusion_r_1_to_ovoid}. 
  
  Now, assume $r\ge 2$ and conclude that for $0\le b'\le q-2$, $0\le a'\le r-2$ we have 
  $n'\notin \left[(a'(q-1)+b')\gauss{r}{1}{q}+a'+1,(a'(q-1)+b'+1)\gauss{r}{1}{q}-1\right]$ for the cardinality 
  $n'$ of a $q^{r-1}$-divisible set. If $a\le r-2$ and $x\le \gauss{r}{1}{q}-1$, then $b'=b$, $a'=a$ yields 
  $\mathcal{T}(\mathcal{C})\subseteq \{u,u+\Delta,\dots,u+(m-2)\Delta\}$ for $u=\Delta+(m-1)\gauss{r}{1}{q}+x$. We compute
  $
    (q-1)u=q^{r+1}-q^r+(m-1)q^r-(m-1)+(q-1)x\overset{x\ge a+1}{\ge} (m-2)q^r+q^{r+1}>(m-2)\Delta
  $,
  so that we can apply Corollary~\ref{cor_nonexistence_arithmetic_progression_2}. If $a=r-1$ and $a+1\le x\le \gauss{r}{1}{q}-1$, 
  then $b'=b$, $a'=a-1$ yields $\mathcal{T}(\mathcal{C})\subseteq \{u,u+\Delta,\dots,u+(m-1)\Delta\}$ for 
  $u=(m-1)\gauss{r}{1}{q}+x$. We compute $(q-1)u=(m-1)q^r-(m-1)+x(q-1)> (m-1)\Delta$ using $x\ge a+1$, so that we can 
  apply Corollary~\ref{cor_nonexistence_arithmetic_progression_2}. Thus, we can assume $\gauss{r}{1}{q}\le x\le \gauss{r+1}{1}{q}-1$ 
  in the remaining part. Additionally we have $m\le r(q-1)$.
  
  We aim to apply Lemma~\ref{lemma_negative_tau}. Due to Lemma~\ref{lemma_m_bound_rhs} for the upper bound of the interval it suffices 
  to show  
  $
    r(q-1)\le  \left\lfloor\sqrt{(q-1)q\Delta}-q+\frac{3}{2}\right\rfloor
  $.
  For $q=2$ the inequality is equivalent to $r\le \left\lfloor\sqrt{2^{r+1}}-\frac{1}{2}\right\rfloor$, which is valid for $r\ge 2$. 
  Since the right hand side is larger then $(q-1)(\sqrt{\Delta}-1)$, it suffices to show $q^{r/2}-1\ge r$, which is valid for $q\ge 3$ 
  and $r\ge 2$. For the left hand side of the interval if suffices to show
  $$
    (q-1)n-(m-1/2)\Delta q+\frac{1}{2}\ge -\frac{1}{2}\cdot\sqrt{(\Delta q)^2-(4m-2)\Delta q+1},
  $$
  which can be simplified to 
  $
    \Delta q+2m-3-2(q-1)x\le \sqrt{(\Delta q)^2-(4m-2)\Delta q+1}
  $
  using $n=(m-1)\gauss{r+1}{1}{q}+x$. Since $(q-1)x\ge q^r-1$ and $m\le r(q-1)$ it suffices to show
  \begin{equation}
    \label{ie_numeric_1}
    -\Delta^2+2rq\Delta-2r\Delta-\Delta-r+r^2q-r^2\le 0.
  \end{equation} 
  For $q=2$ this inequality is equivalent to $-2^{2r}+r2^{r+1}+r^2-2-2^r\le 0$, which is valid for $r\ge 2$. For $r=2$ 
  Inequality~(\ref{ie_numeric_1}) is equivalent to $-q^4+4q^3-4q^2-q^2+4q-6$, which is valid for $q\in\{2,3\}$ and $q\ge 4$. 
  For $q\ge 3$ and $r\ge 3$ we have $\Delta\ge 3rq$, so that Inequality~(\ref{ie_numeric_1}) is satisfied. 
\qed\end{proof}

This classification result enables us to decide the existence
problem for $q^r$-divisible sets over $\F_q$ of cardinality $n$ in
many further cases. We restrict ourselves to the cases $q=2$,
$r\in\{1,2,3\}$, and refer to \cite{dsmta:q-r-divisble} for further
results. 


\begin{theorem}
  \label{thm:picture}
  \begin{enumerate}[(i)]
  \item\label{picture_q_2_r_1} $2^1$-divisible sets over $\F_2$ of
    cardinality $n$ exist for all $n\geq 3$ and do not exist for
    $n\in\{1,2\}$; in particular, $\frobenius(2,1)=2$.
  \item\label{picture_q_2_r_2} $2^2$-divisible sets over $\F_2$ of
    cardinality $n$ exist for $n\in\{7,8\}$ and all $n\geq 14$, and do
    not exist in all other cases;
    in particular,
    $\frobenius(2,2)=13$.
  \item\label{picture_q_2_r_3} $2^3$-divisible sets over $\F_2$ of
    cardinality $n$ exist for
    $$n\in\{15,16,30,31,32,45,46,47,48,49,50,51\},$$ for all $n\geq 60$,
    and possibly for $n=59$; in all other cases they do not exist;
    thus $\frobenius(2,3)\in\{58,59\}$.
  \end{enumerate}
\end{theorem}




\begin{proof}
  \eqref{picture_q_2_r_1} The non-existence for $n\in\{1,2\}$ is
  obvious. Existence for $n\geq 3$ can be shown by taking
  $\mathcal{C}$ as a projective basis in $\PG(n-2,\F_q)$. The
  corresponding code $C$ is the binary $[n,n-1,2]$ even-weight code.

  \eqref{picture_q_2_r_2} The non-existence part follows from
  Theorem~\ref{thm_exclusion_q_r}. Existence for $n\in\{7,8\}$ is
  shown by the $[7,3,4]$ simplex code and its dual (the $[8,4,4]$
  extended Hamming code). These two examples and Lemma~\ref{lemma_sum}
  in turn yield examples of $4$-divisible point sets for
  $n\in\{14,15,16\}$.\footnote{The three examples are realized in
    dimensions $6$, $7$ and $8$, respectively. Alternative solutions
    for $n\in\{15,16\}$, having smaller ambient space dimensions, are
    the $[15,4,8]$ simplex code and the $[16,5,8]$ first-order
    Reed-Muller code.}  For $n\in\{17,18,19,20\}$
  Lemma~\ref{lma:construction1} provides examples.\footnote{These
    examples can be realized in $\F_2^6$ for $n\in\{17,18\}$ and in
    $\F_2^7$ for $n\in\{19,20\}$.} Together these represent all
  arithmetic progressions modulo $7$, showing existence for $n>20$.  

  \eqref{picture_q_2_r_3} Existence of $8$-divisible sets for the
  indicated cases with $n\leq 48$ is shown in the same way as in
  \eqref{picture_q_2_r_2}.  Examples for $n=49$, $n=50$, and $n=74$ were found
  by computer search; we refer to \cite{dsmta:q-r-divisble} for
  generator matrices of the corresponding $8$-divisible codes. 
  The binary irreducible cyclic $[51,8]$ code, which is a two-weight
  code with nonzero weights $24$ and $32$ (see, e.g.,
  \cite{macwilliams-seery}), provides an example for $n=51$.\footnote{It might look tempting to construct a projective $8$-divisible binary code of length $50$ by shortening such a code
  $C$ of length $51$. However, this does not work:
  By Lemma~\ref{lemma_classification_n_51}, $C$ is the concatenation of an ovoid in $\PG(3,\F_4)$ with the binary $[3,2]$ simplex code.
  By construction, the corresponding $8$-divisible point set $\mathcal{C}$ is the disjoint union of $17$ lines.
  In particular, each point of $\mathcal{C}$ is contained in a line in
  $\mathcal{C}$. Consequently, shortening $C$ in any coordinate never
  gives a projective code.}

  For $n\in\{63,\dots,72\}$ Lemma~\ref{lma:construction1} provides examples. 
  For $n=73$, a suitable example is given by the projective $[73,9]$ two-weight code with non-zero weights $32$ and $40$ in \cite{kohnert2007constructing_twoweight}. 
  Together with the mentioned example for $n=74$ these represent all
  arithmetic progressions modulo $15$, showing existence for $n>74$.
  
  The non-existence part follows for $n\leq 48$ from
  Theorem~\ref{thm_exclusion_q_r} and for $53\leq n\leq 58$ from
  Lemma~\ref{lemma_hyperplane_types_arithmetic_progression_c} with
  $m=4$.
  It remains to exclude an $8$-divisible point set in
  $\aspace$ with $\card{\mathcal{C}}=52$. For this we will use a
  variant of the linear programming method, which treats different
  ambient space dimensions simultaneously. Since $\mathcal{C}$ is in
  particular $4$-divisible, we conclude from
  Lemma~\ref{lemma_heritable} and Part~\eqref{picture_q_2_r_2} that
  there are no $4$- or $12$-hyperplanes, i.e., $A_{40}=A_{48}=0$.
  Using the parametrization $y=2^{v-3}$, the first four MacWilliams
  identities for the associated code $C$ are
  \begin{equation*}
    \begin{array}{rrrrrrrrrrl}
    1&+&A_8&+&A_{16}&+&A_{24}&+&A_{32} &=& 8y,\\
    52&+&44A_8&+&36A_{16}&+&28A_{24}&+&20A_{32} &=& 4y\cdot 52,\\
    {52\choose 2}&+&{44\choose 2}A_8&+&{36\choose 2}A_{16}&+&{28\choose 2}A_{24}
    &+&{20 \choose 2}A_{32} &=& 2y\cdot{52 \choose 2},\\
    {52 \choose 3}&+&{44 \choose 3} A_8&+&{36 \choose 3} A_{16}
    &+&{28 \choose 3} A_{24}&+&{20 \choose 3} A_{32} &=& y\left({52\choose
                                                     3}+A_3^\perp\right).         
    \end{array}                                                 
 \end{equation*} 
  Substituting $x=yA_3^\perp$ and solving for $A_8$, $A_{16}$,
  $A_{24}$, $A_{32}$ yields  
    $A_8 = -4+\frac{1}{512} x+\frac{7}{64}y$,
    $A_{16} = 6-\frac{3}{512} x-\frac{17}{64}y$,
    $A_{24} = -4+ \frac{3}{512} x+\frac{397}{64}y$,  and 
    $A_{32} = 1- \frac{1}{512} x+\frac{125}{64}y$.
  Since $A_{16},x\ge 0$, we have $y\le \frac{384}{17}<23$.
  On the other hand, since $3A_8+A_{16}\ge 0$, 
  we also have $-6+\frac{y}{16}\ge 0$, i.e., $y\ge 96$---a contradiction.
  \qed
\end{proof}

The non-existence of a $2^3$-divisible set of cardinality $n=52$
implies the (compared with Theorem~\ref{thm_partial_spread_4}) tightened 
upper bound $\smax_2(11,8;4)\le 132$\footnote{Consequently, for all $t\geq 2$ the
  upper bound for $\smax_2(4t+3,8;4)$ is tightened by one; cf.\
  Lemma~\ref{lemma_extend_construction}.}
and can also be obtained from a more general result, viz.,
Corollary~\ref{cor_implication_fourth_mac_williams} with
$t=3$. Combining the first four MacWilliams identities we obtain an
expression involving a certain cubic polynomial
\cite{dsmta:q-r-divisble}:

\begin{lemma}
  \label{lemma_implication_fourth_mac_williams}
  Let $\mathcal{C}$ be $\Delta$-divisible over $\F_q$ of cardinality
  $n>0$ and $t\in\mathbb{Z}$. Then
  $ \sum_{i\ge 1} \Delta^2(i-t)(i-t-1)\cdot (g_1\cdot i+g_0)\cdot
  A_{i\Delta}\,\,+qhx = n(q-1)(n-t\Delta)(n-(t+1)\Delta)g_2 $, where
  $g_1=\Delta qh$, $g_0=-n(q-1)g_2$,
  $g_2=h-\left(2\Delta qt+\Delta q-2nq+2n+q-2\right)$ and
  $ h= \Delta^2q^2t^2+\Delta^2q^2t-2\Delta nq^2t-\Delta nq^2+2\Delta
  nqt+n^2q^2+\Delta nq-2n^2q+n^2+nq-n $.
\end{lemma}

\begin{corollary}
  \label{cor_implication_fourth_mac_williams}
  If there exists $t\in\mathbb{Z}$, 
  using the notation of
  Lemma~\ref{lemma_implication_fourth_mac_williams}, with
  $n/\Delta\notin [t,t\!+\!1]$, $h\ge 0$, and $g_2<0$, then 
  there 
  is 
  no $\Delta$-divisible 
  set over $\F_q$ of cardinality $n$.
\end{corollary}
Applying Corollary~\ref{cor_implication_fourth_mac_williams} with $t=6$ implies 
the non-existence of a $2^4$-divisible set $\mathcal{C}$ over $\F_2$ with 
$\card\mathcal{C}=200$, so that $\smax_2(16,12;6)\le 1032$, while Theorem~\ref{main_theorem_2} 
gives $\smax_2(16,12;6)\le 1033$. There is no $8^5$-divisible set $\mathcal{C}$ over $\F_8$ with 
$\card\mathcal{C}=6\cdot 8^6+3=1572867$, which can be seen as follows. Corollary~\ref{corollary_iterated_average} 
implies the existence of a subspace $U$ of co-dimension $4$ such that $\mathcal{C}\cap U$ is $8^1$-divisible, 
$\card\mathcal{C}\cap U\le 2\cdot 8^2+3=131$, and $\card\mathcal{C}\cap U\equiv 3\pmod {8^2}$. 
Applying Lemma~\ref{lemma_hyperplane_types_arithmetic_progression_c} with $m=1$ and $m=8$ excludes the existence 
of a $8^1$-divisible set with cardinality $3$ or $67$, respectively. Cardinality $131$ is finally excluded by 
Corollary~\ref{cor_implication_fourth_mac_williams} with $t=14$. Thus, $\smax_8(14;12;6)\le 16777237$, while 
Theorem~\ref{main_theorem_2} gives $\smax_8(14,12;6)\le 16777238$ 
and Theorem~\ref{thm_partial_spread_4} gives $\smax_8(14,12;6)\le 16777248$. See also 
\cite{TableSubspacecodes,kurzspreadsII} for a few more such examples.

Most of the currently best known bounds for $\smax_q(v,2k;k)$ can also be directly 
derived by linear programming; cf.\ the online tables at
\url{http://subspacecodes.uni-bayreuth.de} \cite{TableSubspacecodes}. The following lemma 
gives a glimpse on coding-theoretic arguments dealing with the MacWilliams equations and 
its non-negative integer solutions.

\begin{lemma}
  \label{lemma_classification_n_51}
  Let $C$ be a projective $8$-divisible binary code of length $n=51$.
  Then $C$ is isomorphic to the concatenation of an ovoid in $\PG(3,\F_4)$ with the
  binary $[3,2]$ simplex code.
  The code $C$ has the parameters $[51,8]$ and the weight enumerator $1 + 204X^{24} + 51 X^{32}$.
\end{lemma}
\begin{proof}
  With the notation $k=\dim(C)$ and $y=2^{k-3}$, solving the equation system of the first three MacWilliams equations 
  yields $A_0=1$,
  $
    A_{16} = -6-3A_8+\frac{3}{16}y$,
    $A_{24} = 8+3A_8+\frac{49}{8}y$, and 
    $A_{32} = -3-A_8+\frac{27}{16}y$. 
  Since $A_{16}\ge 0$, we have $y\ge 32$ and hence $k\ge 8$.
  Plugging the stated equations into the fourth MacWilliams equation and solving for $A_8$ gives $A_8=\frac{y A_3^\perp}{512}+\frac{47 y}{512}-4$ and $A_{16}=6-\frac{3yA_3^\perp}{512}-\frac{45y}{512}$.
  Since $A_{16}\ge 0$ and $yA_3^\perp\ge 0$, we have $6-\frac{45y}{512}\ge 0$, so that $y\le 68+\frac{4}{15}$ and therefore $k\le 9$.

  For $k=9$, i.e., $y=64$, $A_{16}\ge 0$ gives $A_3^\perp\le 1$. 
  $A_3^\perp=0$ leads to $A_8=\frac{15}{8}$, which is impossible.
  For $A_3^\perp=1$ the resulting weight enumerator of $C$ is $1 + 2X^8 + 406 X^{24} + 103 X^{32}$.
  However, there is no such code, as the sum of the two codewords
  of weight $8$ would be a third non-zero codeword of weight at most $16$, which does not exist.

  In the case $k=8$, i.e., $y = 32$, the first MacWilliams equation forces $A_{16}=A_8=0$. 
  The resulting weight enumerator of $C$ is given by $1 + 204X^{24} + 51 X^{32}$.
  In particular, $C$ is a projective $[51,8]$ two-weight code.
  By \cite{bouyukliev2006projective}, this code is unique.
  The proof is concluded by the observation that an ovoid in $\PG(3,\F_4)$ is a projective quaternary $[17,4]$ two-weight code, such that the concatenation with the binary $[3,2]$ simplex code yields a projective binary $[51,8]$ two-weight code.
\qed
\end{proof}

\section{Open research problems}
\label{sec_open_problems}

In this closing section we have collected some open research problems
within the scope of this article. All of them presumably are
accessible using the theoretical framework of $q^r$-divisible
sets. Considerably more challenging is the question whether similar
methods can be developed for arbitrary constant-dimension codes in
place of of partial spreads (or vector space partitions). We only
mention the following example: The proof of $\smax_2(6,4;3)=77$ still
depends on extensive computer calculations providing the upper bound
$\smax_2(6,4;3)\leq 77$ \cite{hkk77}. The known theoretical upper
bound of $81$ may be sharpened to $77$ (along the lines
of~\cite{nakic2016extendability}), if only the existence of a
$(6,81,4;3)_2$ code can be excluded. The $81$ planes in such a code
would form an exact $9$-cover of the line set of $\PG(5,\F_2)$.

\subsection{Better constructions for partial
  spreads}\label{ssec:better}

The only known cases in which the construction of
Theorem~\ref{thm:multicomponent} has been surpassed are derived from the
sporadic example of a partial $3$-spread of cardinality $34$ in
$\mathbb{F}_2^8$ \cite{spreadsk3}, which has $17$ holes and can
be used to show $\smax_2(3m+2,6;3)\geq(2^{3m+2}-18)/7$ by adding $m-2$
layers of lifted MRD codes; cf.\
Lemma~\ref{lemma_extend_construction} and
Corollary~\ref{cor_partial_spread_k3}.
A first step towards the understanding of
the sporadic example is the classification of all $2^2$-divisible
point sets of cardinality $17$ in $\PG(k-1,\F_2)$. It turns out that
there are exactly $3$ isomorphism types, one configuration
$\mathcal{H}_k$ for each dimension
$k\in\{6,7,8\}$. Generating matrices for the corresponding
doubly-even codes are
\begin{equation}
  \label{eq:17gen}
  \footnotesize
\left(\begin{array}{c}
10000110010101110\\
01000010111011100\\
00100100000011000\\
00010111001110100\\
00001001100111110\\
00000011100111011\\
\end{array}\right),
\;
\left(\begin{array}{c}
10000011110100110\\
01000001111111000\\
00100010000110000\\
00010010000101000\\
00001001001000100\\
00000101001000010\\
00000010101011111\\
\end{array}\right),
\;
\left(\begin{array}{c}
10000000111011110\\
01000000010110000\\
00100000011100000\\
00010000001110000\\
00001001100000010\\
00000101000001010\\
00000011000000110\\
00000001111011101\\
\end{array}\right).
\normalsize
\end{equation}

While the classification, so far, is based on computer
calculations,\footnote{See \url{http://www.rlmiller.org/de\_codes} and
  \cite{doran2011codes} for the classification of, possibly
  non-projective, doubly-even codes over $\mathbb{F}_2^v$.}  one can
easily see that there are exactly three solutions of the MacWilliams
identities.
\begin{lemma}
  Let $\mathcal{C}$ be a $2^2$-divisible point set over $\F_2$
  of cardinality $17$. Then
  $k=\dim\langle\mathcal{C}\rangle\in\{6,7,8\}$, and the  
  solutions of the MacWilliams identities are unique for each $k$: 
  $(k;a_5,a_9,a_{13};A_3^\perp)=(6;12,49,2;6)$, $(7;25,95,7;2)$,
  $(8;51,187,17;0)$.
\end{lemma}
\begin{proof}
  The unique solution of the standard equations
  is given by $a_5=\frac{13}{16}\cdot 2^{k-2}-1$,
  $a_9=\frac{23}{8}\cdot 2^{k-2}+3$, and
  $a_{13}=\frac{5}{16}\cdot 2^{k-2}-3$. Hence $k\ge 6$, since
  otherwise $a_{13}<0$, and $k\leq 8$, since $C$ is
  self-orthogonal.\footnote{Alternatively, the 4th MacWilliams
  identity yields $64-2^{k-2}=2^{k-3}\cdot A_3^\perp$ and hence
  $k\le 8$.}
\qed\end{proof}
The hole set of the partial $3$-spread in \cite{spreadsk3} corresponds to
$\mathcal{H}_7$. A geometric description
of this configuration is given in
\cite[p.~84]{lambert2013random}. We have computationally checked that
indeed all three hole configurations can be realized by a partial $3$-spread
of cardinality $34$ in $\mathbb{F}_2^8$.\footnote{$624$ non-isomorphic
  examples can be found at
  \url{http://subspacecodes.uni-bayreuth.de}
  \cite{TableSubspacecodes}.}  All three configurations have
a non-trivial automorphism group, and hence there is a chance to
find a partial $3$-spread with nontrivial symmetries and eventually discover
an underlying more general construction.%
\footnote{In a forthcoming
  paper we classify the $2612$ non-isomorphic partial $3$-spreads of
  cardinality $34$ in $\mathbb{F}_2^8$ that admit an automorphism
  group of order exactly $8$, which is possible for $\mathcal{H}_6$ only, 
  and show that the automorphism groups of all other examples have order at most $4$.}
  So far we can
only describe the geometric structure of 
$\mathcal{H}_6$, $\mathcal{H}_7$, $\mathcal{H}_8$.
The hole configuration $\mathcal{H}_6$ consists of two
disjoint planes $E_1,E_2$ in $\PG(5,\F_2)$ and a solid $S$ spanned by
lines $L_1\subset E_1$ and $L_2\subset E_2$.  The $17=4+4+9$ points of
$\mathcal{H}_6$ are those in $E_1\setminus L_1$, $E_2\setminus L_2$,
and $S\setminus(L_1\cup L_2)$. An application of
Lemma~\ref{lma:construction1} (the case $q=r=i=2$) gives that
$\mathcal{H}_6$ is $4$-divisible.  In sunflower terms, 
$\mathcal{F}_1=\{S,E_1\}$, $\mathcal{F}_2=\{S,E_2\}$ with centers
$L_1$, $L_2$, respectively.

The hole configuration $\mathcal{H}_7$ can be obtained by modifying a
$3$-sunflower $\mathcal{F}=\{E,S_1,S_2\}$, whose petals are a plane
$E$ and two solids $S_1,S_2$ and whose base is a line $L$. By
Lemma~\ref{lma:sunflower}(ii), the point set $E\cup S_1\cup S_2$, of
cardinality $3+4+12+12=31$ is $4$-divisible. From this set we remove
two planes $E_1\subset S_1$, $E_2\subset S_2$ intersecting $L$ in
different points. This gives $\mathcal{H}_7$.

The hole configuration $\mathcal{H}_8$ can be obtained by modifying
the cone construction of Lemma~\ref{lma:cone}. We start with a
projective basis $\mathcal{B}$ in $\F_2^4$, i.e., $m=5$ points with no
$4$ of them contained in a plane. Such point sets $\mathcal{B}$ are
associated to the $[5,4,2]_2$ even-weight code and hence
$2$-divisible. Since $m\equiv 1\pmod{4}$, the proof of
Lemma~\ref{lma:cone} shows that a generalized cone $\mathcal{K}$ over
$\mathcal{B}$ with $1$-dimensional vertex $Q$ of multiplicity
$-m(q-1)\equiv-m\equiv 3\pmod{4}$ is $4$-divisible. Working over
$\mathbb{Z}$, we can set $\chi_{\mathcal{K}}(Q)=-1$ as well. Adding
any $4$-divisible point set $\mathcal{D}$ containing $Q$ (i.e.,
$\chi_{\mathcal{D}}(Q)=1$) to $\mathcal{K}$ then produces a
$4$-divisible multiset set $\mathcal{C}$ with
$\card{\mathcal{C}}=10-1+\card{\mathcal{D}}$ and
$\chi_{\mathcal{C}}(Q)=0$. By making the ambient spaces of
$\mathcal{K}$ and $\mathcal{C}$ intersect only in $Q$ (which requires
embedding the configuration into a larger space $\F_2^v$), we can
force $\mathcal{C}$ to be a set. The configuration $\mathcal{H}_8$ is
obtained by choosing $\mathcal{D}$ as an affine solid.\footnote{Since
  punctured affine solids are associated to the $[7,4,3]_2$ Hamming
  code, we may also think of $\mathcal{H}_8$ as consisting of the
  $2$-fold repetition of the $[5,4,2]$-code and the Hamming code ``glued
  together'' in $Q$. In fact the doubly-even $[17,8,4]_2$ code
  associated with 
  $\mathcal{H}_8$ is the code $\overline{I}_{17}^{(3)}$ in
  \cite[p.~234]{pless1975classification}. The glueing construction
  is visible in the generator matrix.}

From the
preceding discussion it is clear that all possible hole types of a
partial $3$-spread of cardinality $34$ in $\mathbb{F}_2^8$ belong to
infinite families of $q^r$-divisible sets.  Can further
$2^2$-divisible sets of small cardinality be extended to an infinite
family?

\begin{Construction}
  \label{construction_4}
  For integers $r\ge 1$ and $0\le m\le r$ let $S$ be a $2r$-subspace. By $F_1,\dots,F_{\gauss{m}{1}{q}}$ we denote 
  the $(2r+1)$-subspaces of $\mathbb{F}_{2r+m}$ that contain $S$ and by $L_1,\dots,L_{\gauss{m}{1}{q}}$ we denote 
  a list of $r$-subspaces of $S$ with pairwise trivial intersection. Let $0\le a\le \gauss{m}{1}{q}$ and $0\le b_i\le q^{r-1}-1$ 
  for all $1\le i\le a$. For each $1\le i\le a$ we choose $q-1+b_iq=:c_i$ different $(r+1)$-subspaces $E_{i,j}$ of $F_i$ with $F_i\cap S=L_i$. 
  With this, we set $\mathcal{C}=\left(S\backslash \cup_{i=1}^a L_i\right)\,\cup\, \left(\cup_{i=1}^a \cup_{j=1}^{c_i} \left(E_{i,j}\backslash L_i\right)\right)$ 
  and observe $\dim(\mathcal{C})\le 2r+m$, $\card{\mathcal{C}}=\gauss{2r}{1}{q}+a\cdot\left(q^{r+1}-\gauss{r+1}{1}{q}\right)+b\cdot q^{r+1}$, 
  where $b=\sum_{i=1}^{a}b_i$, and that $\mathcal{C}$ is $q^r$-divisible.    
\end{Construction}
\begin{proof}
Apply Lemma~\ref{lemma_subspace}, \ref{characteristic sum} using $\chi_{\mathcal{C}}^v=\chi_{S}^v+\sum\limits_{i=1}^a\sum\limits_{j=1}^{c_i} \chi_{E_{i,j}}^v-q\sum\limits_{i=1}^a (b_i+1)\chi_{L_i}^v$.
\qed\end{proof}
We remark that the construction can easily be modified to obtain $q^r$-divisible sets of cardinality 
$n=\gauss{2r}{1}{q}+a\cdot\left(q^{r+1}-\gauss{r+1}{1}{q}\right)+b\cdot q^{r+1}$ and dimension $v$ for all 
$2r+m \le v\le 2r+a(q-1)+bq$. Choosing $m=r$, $a=q^{r-1}$, and $b=0$ we obtain a $q^r$-divisible set $\mathcal{C}$ of cardinality $n=q^{2r}+\gauss{r-1}{1}{q}$ 
and dimension $v=3r$. For which parameters $r$ and $q$ do partial $(r+1)$-spreads of cardinality $q^{2r+1}+q^{r-1}$, with 
hole configuration $\mathcal{C}$,  
in $\mathbb{F}_q^{3r+2}$ exist?  
So far, such partial spreads are known for $r=1$, $(q,r)=(2,2)$ and all further examples would be strictly larger than the ones from 
Theorem~\ref{thm:multicomponent}.  

\medskip

For the corresponding parameters over the ternary field we currently
only know the bounds $244\le \smax_3(8,6;3)\le 248$. A putative plane
spread in $\F_3^8$ of size $248$ would have a $3^2$-divisible hole
configuration $\mathcal{H}$ of cardinality $56$. Such a point set is
unique up to isomorphism and has dimension $k=6$. It corresponds to an
optimal two-weight code with weight distribution
$0^1 36^{616} 45^{112}$. The set $\mathcal{H}$ was first described by R.~Hill in
\cite{hill1978caps} and is known as the \emph{Hill cap}. A generator
matrix for 
is
\begin{equation*}
  \footnotesize
  \setlength{\arraycolsep}{0.05em}
  \left(
  \begin{array}{cccccccccccccccccccccccccccccccccccccccccccccccccccccccc}
    1&0&0&0&0&0&2&2&1&1&0&1&0&0&1&1&0&2&0&2&1&1&1&1&0&0&1&0&1&2&0&1&0&2&1&2&1&1&1&1&1&2&2&0&0&1&2&0&0&2&0&1&2&2&1&1\\
    0&1&0&0&0&0&1&1&1&0&1&2&1&0&1&0&1&1&2&1&1&2&0&0&1&0&2&1&1&2&2&2&1&1&1&2&1&0&0&0&0&2&1&2&0&2&2&2&0&0&2&2&2&0&1&0\\
    0&0&1&0&0&0&2&2&2&2&0&2&2&1&0&2&0&0&1&1&2&0&0&1&0&1&1&2&0&0&2&0&2&0&2&0&0&2&1&1&1&2&2&1&2&1&1&2&2&2&0&0&1&1&1&2\\
    0&0&0&1&0&0&1&0&1&1&2&2&2&2&0&2&2&1&0&2&0&0&2&2&1&0&0&1&0&1&0&1&0&0&2&2&2&2&1&0&0&2&2&2&1&1&2&1&2&2&2&2&1&2&0&0\\
    0&0&0&0&1&0&2&0&1&2&1&0&2&2&1&1&2&1&1&2&0&0&1&0&2&1&1&0&2&2&1&1&1&2&1&0&0&0&0&2&1&2&0&2&2&2&0&2&1&2&2&0&1&0&0&1\\
    0&0&0&0&0&1&1&2&2&0&2&0&0&2&2&0&1&0&1&2&1&2&2&0&0&2&0&1&1&0&2&0&1&2&1&2&2&2&2&2&1&2&0&0&2&1&0&0&2&0&2&1&1&2&2&2
  \end{array}\right).
\end{equation*}
The automorphism group 
has order $40320$. 
Given the large automorphism
group of $\mathcal{H}$, is it possible to construct a partial
plane spread in $\mathbb{F}_3^8$ with size larger than $244$?

\medskip

For $q=2$, the first open case is $129\le \smax_2(11,8;4)\le 132$. A
putative partial $4$-spread of size 132 has a $2^3$-divisible hole
configuration of cardinality $67$. Such exist for all dimensions $9\le
k\le 11$; cf.\ Theorem~\ref{thm:picture}\eqref{picture_q_2_r_3}.
Can one such $2^3$-divisible set be completed to a partial $4$-spread?

\medskip

Already the smallest open cases pose serious computational
challenges. A promising approach is the prescription of automorphisms
in order to reduce the computational complexity; see,
e.g., \cite{kohnert2008construction} for an application of this
so-called \emph{Kramer-Mesner method} to constant-dimension codes. Of
course, the automorphisms have to stabilize the hole configuration,
whose automorphism group is known or can be easily computed in many
cases.

 \subsection{Existence results for $\bm{q^r}$-divisible sets}

 Even for rather small parameters $q$ and $r$ we 
 cannot decide the existence question, 
 see Table~\ref{table_open_cases}.\\[-7mm]

 \begin{table}
   \begin{center}
     \begin{tabular}{rrl}
       \hline
       $q$ & $r$ & $n$ \\
       \hline
       2 & 3 & 59 \\
       2 & 4 & 
             130, 131, 163, 164, 165, 185, 215, 216, 232, 233, 244, 245, 246, 247 \\
       3 & 2 & 70, 77, 99, 100, 101, 102, 113, 114, 115, 128\\
       4 & 2 & 129, 150, 151, 172, 173, 193, 194, 195, 215, 216, 217, 236, 237,238, 239, 251, 258,\\  
         &   & 259, 261, 272, 279, 280, 282, 283, 293, 301, 305, 313, 314, 322, 326, 333, 334, 335,\dots \\
       5 & 1 & 40\\
       7 & 1 & 75, 83, 91, 92, 95, 101, 102, 103, 109, 110, 111, 117, 118, 119, 125, 126, 127, 133, \\
         &   & 134, 135, 142, 143, 151, 159, 167 \\  
       8 & 1 & 93, 102, 111, 120, 121, 134, 140, 143, 149, 150, 151, 152, 158, 159, 160, 161, 167,\\
         &   & 168, 169, 170, 176, 177, 178, 179, 185, 186, 187, 188, 196, 197, 205, 206, 214, 215,\\
         &   & 223, 224, 232, 233, 241, 242, 250, 251\\  
       9 & 1 & 123, 133, 143, 153, 154, 175, 179, 185, 189, 195, 196, 199, 206, 207, 208, 209, 216, \\
         &   & 217, 218, 219, 226, 227, 228, 229, 236, 237, 238, 239, 247, 248, 249, 257, 258, 259, \\
         &   & 267, 268, 269, 277, 278, 279, 288, 289, 298, 299, 308, 309, 318, 319, 329, 339, 349,\\
         &   & 359\\  
       \hline
     \end{tabular} 
     \caption{Undecided cases for the existence of $q^r$-divisible sets.}
     \label{table_open_cases}     
   \end{center}
 \end{table}

\subsection{Vector space partitions}

The most general result on the non-existence of vector space
partitions (without conditions on the ambient space dimension $v$)
seems to be:

\begin{theorem} (Theorem 1 in \cite{heden2009length})
  \label{thm_length_of_tail}
  Let $\mathcal{C}$ be a vector space partition of type
  $k^{z}\dotsm {d_2}^{b}{d_1}^{a}$ of $\SV{v}{q}$, where
  $a,b>0$.
  \begin{enumerate}[(i)]
  \item If $q^{d_2-d_1}$ does not divide $a$ and if $d_2<2d_1$, then
    $a\ge q^{d_1}+1$;
  \item if $q^{d_2-d_1}$ does not divide $a$ and if
    $d_2\ge 2d_1$, then $a>2q^{d_2-d_1}$ or $d_1$ divides $d_2$ and
    $a=\left(q^{d_2}-1\right)/\left(q^{d_1}-1\right)$;
  \item if $q^{d_2-d_1}$ divides $a$ and $d_2<2d_1$, then
    $a\ge q^{d_2}-q^{d_1}+q^{d_2-d_1}$;
  \item if $q^{d_2-d_1}$ divides $a$ and $d_2\ge 2d_1$, then
    $a\ge q^{d_2}$.
  \end{enumerate}   
\end{theorem}

For the special case $d_1=1$,
Theorems~\ref{thm_exclusion_r_1_to_ovoid} and~\ref{thm_exclusion_q_r},
presented in Section~\ref{sec_non_existence}, provide tighter
results. For $d_1>1$ we can replace $d_1$-subspaces by the
corresponding point sets and apply results for $q^r$-divisible
sets. For vector space partitions of type $k^z\dotsm 4^b2^{a}$ in
$\SV{v}{2}$ we obtain $2^3$-divisible sets of cardinality $n=3a$, so
that $a\in\{5,10,15,16\}$ or $a\ge 20$ by
Theorem\ref{thm:picture}\eqref{picture_q_2_r_3}. 
Theorem~\ref{thm_length_of_tail}
gives $a=5$ or $a\ge 9$, and $4\mid a$ implies
$a\ge 16$. However, not all results of
Theorem~\ref{thm_length_of_tail} can be obtained that easy. For vector
space partitions of type $k^z\dotsm 4^b3^a$ in $\SV{v}{2}$ we
obtain $2^3$-divisible sets of cardinality $n=7a$, giving
$a=7$ or $a\ge
9$. Theorem~\ref{thm_length_of_tail} gives $a\ge 9$, and
$2\mid a$ implies $a\ge 10$.  In order to exclude
$a=7$ one has to look at hyperplane intersections in this new
setting. So far, we have used
$\card{(H\cap \mathcal{C})}\equiv\card{\mathcal{C}}\pmod {q^r}$. The sets
$H\cap \mathcal{C}$ have to come as a partition of type
$s^d (s-1)^c$, where $c+d=a$. Here the possible values for $c$
are restricted by the cases of $q^{r-1}$-divisible sets admitting the
partition type $(s-1)^c$. This further reduces the possible
hyperplane types, so that eventually the linear programming method can be
applied. In our situation we have
$\mathcal{T}(\mathcal{C})\subseteq\{25,49\}$, which is excluded by
Lemma~\ref{lemma_hyperplane_types_arithmetic_progression_2}. For the
general case we introduce the following notation:
A point set $\mathcal{C}$ in $\aspace$ admits the partition type 
$s^{m_s}\dotsm 1^{m_1}$ if there exists a vector space partition of $\F_q^v$
of type $s^{m_s}\dotsm 1^{m_1}$ that covers the points in $\mathcal{C}$
and no other points. 
In terms of this, we may restate the previous result as ``there is no
$2^3$-divisible set admitting the partition type $3^7$''. However, we
are very far from the generality and compactness of
Theorem~\ref{thm_length_of_tail}. Nevertheless, the sketched approach
seems to be a very promising research direction 
(and a natural extension of the study of $q^r$-divisible sets).


\section*{Acknowledgement}
The authors would like to acknowledge the financial support provided by COST -- 
\emph{European Cooperation in Science and Technology}.
The first author was also supported by the National Natural Science Foundation of China under Grant 61571006.
The third author was supported in part by the grant KU 2430/3-1 -- Integer Linear Programming Models for Subspace Codes and Finite Geometry from the German Research Foundation.

 

\begin{thebibliography}{10}

\bibitem{andre1954nicht}
J.~Andr{\'e}.
\newblock {\"U}ber nicht-desarguessche {E}benen mit transitiver
  {T}ranslationsgruppe.
\newblock {\em Mathematische Zeitschrift}, 60(1):156--186, 1954.

\bibitem{andrews1998theory}
G.~Andrews.
\newblock {\em The theory of partitions}.
\newblock Number~2. Cambridge University Press, 1998.

\bibitem{belov-etal74}
B.~I. Belov, V.~N. Logachev, and V.~P. Sandimirov.
\newblock Construction of a class of linear binary codes achieving the
  {V}arshamov-{G}riesmer bound.
\newblock {\em Problems of Information Transmission}, 10:211--217, 1974.

\bibitem{beule-klein-metsch11}
J.~D. Beule, A.~Klein, and K.~Metsch.
\newblock Substructures of finite classical polar spaces.
\newblock In Beule and Storme \cite{nova2011}.

\bibitem{nova2011}
J.~D. Beule and L.~Storme, editors.
\newblock {\em Current Research Topics in {G}alois Geometry}.
\newblock Nova Science Publishers, 2011.

\bibitem{beutelspacher1975partial}
A.~Beutelspacher.
\newblock Partial spreads in finite projective spaces and partial designs.
\newblock {\em Mathematische Zeitschrift}, 145(3):211--229, 1975.

\bibitem{beutelspacher1978partitions}
A.~Beutelspacher.
\newblock Partitions of finite vector spaces: an application of the {F}robenius
  number in geometry.
\newblock {\em Archiv der Mathematik}, 31(1):202--208, 1978.

\bibitem{bierbrauer2005introduction}
J.~Bierbrauer.
\newblock {\em Introduction to coding theory}.
\newblock 2005.

\bibitem{bonferroni1936teoria}
C.~Bonferroni.
\newblock {\em Teoria statistica delle classi e calcolo delle probabilit\`a}.
\newblock Libreria internazionale Seeber, 1936.

\bibitem{bose1952orthogonal}
R.~Bose and K.~Bush.
\newblock Orthogonal arrays of strength two and three.
\newblock {\em The Annals of Mathematical Statistics}, pages 508--524, 1952.

\bibitem{bouyukliev2006projective}
I.~Bouyukliev, V.~Fack, W.~Willems, and J.~Winne.
\newblock Projective two-weight codes with small parameters and their
  corresponding graphs.
\newblock {\em Designs, Codes and Cryptography}, 41(1):59--78, 2006.

\bibitem{brauer1942problem}
A.~Brauer.
\newblock On a problem of partitions.
\newblock {\em American Journal of Mathematics}, 64(1):299--312, 1942.

\bibitem{calderbank1986geometry}
R.~Calderbank and W.~Kantor.
\newblock The geometry of two-weight codes.
\newblock {\em Bulletin of the London Mathematical Society}, 18(2):97--122,
  1986.

\bibitem{delsarte1972bounds}
P.~Delsarte.
\newblock Bounds for unrestricted codes, by linear programming.
\newblock {\em Philips Res. Rep}, 27:272--289, 1972.

\bibitem{delsarte72}
P.~Delsarte.
\newblock Weights of linear codes and strongly regular normed spaces.
\newblock {\em Discrete Mathematics}, 3:47--64, 1972.

\bibitem{dembowski2012finite}
P.~Dembowski.
\newblock {\em Finite {G}eometries: {R}eprint of the 1968 Edition}.
\newblock Springer Science \& Business Media, 2012.

\bibitem{dodunekov1998codes}
S.~Dodunekov and J.~Simonis.
\newblock Codes and projective multisets.
\newblock {\em The Electronic Journal of Combinatorics}, 5(R37):1--23, 1998.

\bibitem{doran2011codes}
C.~Doran, M.~Faux, S.~Gates, T.~H{\"u}bsch, K.~Iga, G.~Landweber, and
  R.~Miller.
\newblock Codes and supersymmetry in one dimension.
\newblock {\em Advances in Theoretical and Mathematical Physics},
  15(6):1909--1970, 2011.

\bibitem{nets_and_spreads}
D.~Drake and J.~Freeman.
\newblock Partial $t$-spreads and group constructible $(s,r,\mu)$-nets.
\newblock {\em Journal of Geometry}, 13(2):210--216, 1979.

\bibitem{eisfeldt}
J.~Eisfeld and L.~Storme.
\newblock $t$-spreads and minimal $t$-covers in finite projective spaces.
\newblock {\em Lecture notes, Universiteit Gent, 29 pages}, 2000.

\bibitem{spreadsk3}
S.~El-Zanati, H.~Jordon, G.~Seelinger, P.~Sissokho, and L.~Spence.
\newblock The maximum size of a partial $3$-spread in a finite vector space
  over ${G}{F}(2)$.
\newblock {\em Designs, Codes and Cryptography}, 54(2):101--107, 2010.

\bibitem{erdos-rado60}
P.~Erd{\H o}s and R.~Rado.
\newblock Intersection theorems for systems of sets.
\newblock {\em Journal of the London Mathematical Society $\mathrm{(1)}$},
  35(1):85--90, 1960.

\bibitem{etzion2009error}
T.~Etzion and N.~Silberstein.
\newblock Error-correcting codes in projective spaces via rank-metric codes and
  {F}errers diagrams.
\newblock {\em IEEE Transactions on Information Theory}, 55(7):2909--2919,
  2009.

\bibitem{galambos1977bonferroni}
J.~Galambos.
\newblock Bonferroni inequalities.
\newblock {\em The Annals of Probability}, pages 577--581, 1977.

\bibitem{Gordon-Shaw-Soicher-2004-unpublished}
N.~A. Gordon, R.~Shaw, and L.~H. Soicher.
\newblock Classification of partial spreads in {$\operatorname{PG}(4,2)$}.
\newblock \url{www.maths.qmul.ac.uk/~leonard/partialspreads/PG42new.pdf}, 2004.

\bibitem{guang-zhang14}
X.~Guang and Z.~Zhang.
\newblock {\em Linear Network Error Correction Coding}.
\newblock SpringerBriefs in Computer Science. Springer, 2014.

\bibitem{heden1984frobenius}
O.~Heden.
\newblock The {F}robenius number and partitions of a finite vector space.
\newblock {\em Archiv der Mathematik}, 42(2):185--192, 1984.

\bibitem{heden2009length}
O.~Heden.
\newblock On the length of the tail of a vector space partition.
\newblock {\em Discrete Mathematics}, 309(21):6169--6180, 2009.

\bibitem{heden2012survey}
O.~Heden.
\newblock A survey of the different types of vector space partitions.
\newblock {\em Discrete Mathematics, Algorithms and Applications}, 4(1):14p.,
  2012.
\newblock nr. 1250001.

\bibitem{dsmta:q-r-divisble}
D.~Heinlein, T.~Honold, M.~Kiermaier, S.~Kurz, and A.~Wassermann.
\newblock On projective $q^r$-divisible codes.
\newblock In preparation, Oct. 2016.

\bibitem{TableSubspacecodes}
D.~Heinlein, M.~Kiermaier, S.~Kurz, and A.~Wassermann.
\newblock {\em Tables of subspace codes}.
\newblock University of Bayreuth, 2015.
\newblock available at http://subspacecodes.uni-bayreuth.de.

\bibitem{hill1978caps}
R.~Hill.
\newblock Caps and codes.
\newblock {\em Discrete Mathematics}, 22(2):111--137, 1978.

\bibitem{hill92}
R.~Hill.
\newblock Optimal linear codes.
\newblock In C.~Mitchell, editor, {\em Cryptography and Coding {II}}, pages
  75--104. Oxford University Press, 1992.

\bibitem{hong1972general}
S.~Hong and A.~Patel.
\newblock A general class of maximal codes for computer applications.
\newblock {\em IEEE Transactions on Computers}, 100(12):1322--1331, 1972.

\bibitem{hkk77}
T.~Honold, M.~Kiermaier, and S.~Kurz.
\newblock Optimal binary subspace codes of length $6$, constant dimension $3$
  and minimum distance $4$.
\newblock {\em Contemp. Math.}, 632:157--176, 2015.

\bibitem{hkk_partial_plane_spreads}
T.~Honold, M.~Kiermaier, and S.~Kurz.
\newblock Classification of large partial plane spreads in $pg(6,2)$ and
  related combinatorial objects.
\newblock {\em preprint}, 2016.

\bibitem{honold2015constructions}
T.~Honold, M.~Kiermaier, and S.~Kurz.
\newblock Constructions and bounds for mixed-dimension subspace codes.
\newblock {\em Advances in Mathematics of Communication}, 10(3):649--682, 2016.

\bibitem{huffman2010fundamentals}
W.~Huffman and V.~Pless.
\newblock {\em Fundamentals of error-correcting codes}.
\newblock Cambridge University Press, 2010.

\bibitem{huffman-pless03}
W.~C. Huffman and V.~Pless.
\newblock {\em Fundamentals of Error-Correcting Codes}.
\newblock Cambridge University Press, 2003.

\bibitem{hughes-piper73}
D.~R. Hughes and F.~C. Piper.
\newblock {\em Projective Planes}.
\newblock Number~6 in Graduate Texts in Mathematics. Springer, 1973.

\bibitem{johnson-jha-biliotti07}
N.~L. Johnson, V.~Jha, and M.~Biliotti.
\newblock {\em Handbook of Finite Translation Planes}.
\newblock CRC Press, 2007.

\bibitem{koetter-kschischang08}
R.~Koetter and F.~Kschischang.
\newblock Coding for errors and erasures in random network coding.
\newblock {\em IEEE Transactions on Information Theory}, 54(8):3579--3591, Aug.
  2008.

\bibitem{kohnert2007constructing_twoweight}
A.~Kohnert.
\newblock Constructing two-weight codes with prescribed groups of
  automorphisms.
\newblock {\em Discrete Applied Mathematics}, 155(11):1451--1457, 2007.

\bibitem{kohnert2008construction}
A.~Kohnert and S.~Kurz.
\newblock Construction of large constant dimension codes with a prescribed
  minimum distance.
\newblock In {\em Mathematical methods in computer science}, pages 31--42.
  Springer, 2008.

\bibitem{kurzspreads}
S.~Kurz.
\newblock Improved upper bounds for partial spreads.
\newblock {\em Designs, Codes and Cryptography}, Oct. 2016.
\newblock published online on Oct 25, 2016.

\bibitem{kurzspreadsII}
S.~Kurz.
\newblock Upper bounds for partial spreads.
\newblock {\em arXiv preprint 1606.08581}, 2016.

\bibitem{lambert2013random}
L.~Lambert.
\newblock {\em Random network coding and designs over $\mathbb{F}_q$}.
\newblock PhD thesis, Ghent University; Master Thesis, 2013.

\bibitem{ivan00}
I.~N. Landjev.
\newblock Linear codes over finite fields and finite projective geometries.
\newblock {\em Discrete Mathematics}, 213:211--244, 2000.

\bibitem{lavrauw-polverino11}
M.~Lavrauw and O.~Polverino.
\newblock Finite semifields.
\newblock In Beule and Storme \cite{nova2011}, chapter~6, pages 129--157.

\bibitem{lloyd1957binary}
S.~Lloyd.
\newblock Binary block coding.
\newblock {\em Bell System Technical Journal}, 36(2):517--535, 1957.

\bibitem{macwilliams63}
F.~J. MacWilliams.
\newblock A theorem on the distribution of weights in a systematic code.
\newblock {\em The Bell System Technical Journal}, 42(1):79--94, 1963.

\bibitem{macwilliams-seery}
F.~J. MacWilliams and J.~Seery.
\newblock The weight distributions of some minimal cyclic codes.
\newblock {\em IEEE Transactions on Information Theory}, 27:796--806, 1981.

\bibitem{MR2475427}
Z.~Mateva and S.~Topalova.
\newblock Line spreads of {${\rm PG}(5,2)$}.
\newblock {\em J. Combin. Des.}, 17(1):90--102, 2009.

\bibitem{medard-sprintson12}
M.~M{\'e}dard and A.~Sprintson, editors.
\newblock {\em Network Coding: {F}undamentals and Applications}.
\newblock Elsevier, 2012.

\bibitem{nakic2016extendability}
A.~Naki{\'c} and L.~Storme.
\newblock On the extendability of particular classes of constant dimension
  codes.
\newblock {\em Designs, Codes and Cryptography}, 79(3):407--422, 2016.

\bibitem{nastase2016maximumII}
E.~N{\u{a}}stase and P.~Sissokho.
\newblock The maximum size of a partial spread {I}{I}: Upper bounds.
\newblock {\em arXiv preprint 1606.09208}, 2016.

\bibitem{nastase2016maximum}
E.~N{\u{a}}stase and P.~Sissokho.
\newblock The maximum size of a partial spread in a finite projective space.
\newblock {\em arXiv preprint 1605.04824}, 2016.

\bibitem{plackett1946design}
R.~Plackett and J.~Burman.
\newblock The design of optimum multifactorial experiments.
\newblock {\em Biometrika}, 33(4):305--325, 1946.

\bibitem{pless1975classification}
V.~Pless and N.~Sloane.
\newblock On the classification and enumeration of self-dual codes.
\newblock {\em Journal of Combinatorial Theory, Series A}, 18(3):313--335,
  1975.

\bibitem{segre1964teoria}
B.~Segre.
\newblock Teoria di galois, fibrazioni proiettive e geometrie non
  desarguesiane.
\newblock {\em Annali di Matematica Pura ed Applicata}, 64(1):1--76, 1964.

\bibitem{MR2801585}
N.~Silberstein and T.~Etzion.
\newblock {Large constant dimension codes and lexicodes}.
\newblock {\em Adv. Math. Commun.}, 5(2):177--189, 2011.

\bibitem{silva2008rank}
D.~Silva, F.~Kschischang, and R.~Koetter.
\newblock A rank-metric approach to error control in random network coding.
\newblock {\em IEEE Transactions on Information Theory}, 54(9):3951--3967,
  2008.

\bibitem{tsfasman-vladut95}
M.~A. Tsfasman and S.~G. Vl{\u a}du{\c t}.
\newblock Geometric approach to higher weights.
\newblock {\em IEEE Transactions on Information Theory}, 41:1564--1588, 1995.

\bibitem{lint-wilson92}
J.~H. van Lint and R.~M. Wilson.
\newblock {\em A Course in Combinatorics}.
\newblock Cambridge University Press, 1992.

\bibitem{ward1998divisibility}
H.~Ward.
\newblock Divisibility of codes meeting the griesmer bound.
\newblock {\em Journal of Combinatorial Theory, Series A}, 83(1):79--93, 1998.

\bibitem{ward1999introduction}
H.~Ward.
\newblock An introduction to divisible codes.
\newblock {\em Designs, Codes and Cryptography}, 17(1):73--79, 1999.

\bibitem{ward2001divisible_survey}
H.~Ward.
\newblock Divisible codes -- a survey.
\newblock {\em Serdica Mathematical Journal}, 27(4):263p--278p, 2001.

\bibitem{yeung-li-cai06}
R.~W. Yeung, S.-Y.~R. Li, N.~Cai, and Z.~Zhang.
\newblock Network coding theory.
\newblock {\em Foundations and Trends in Communications and Information
  Theory}, 2(4/5):241--381, 2006.

\end{thebibliography}

\end{document}